\pdfoutput=1
\documentclass[12pt]{article}
\usepackage{amsmath,amssymb,color,amsthm,amsfonts}
\usepackage{mathtools}
\usepackage[nottoc,numbib]{tocbibind}
\usepackage[style=alphabetic]{biblatex}
\usepackage{tikz-cd}

\DeclareMathOperator{\crit}{crit}

\DeclareMathOperator{\QCoh}{QCoh}
\DeclareMathOperator{\Op}{Op}

\DeclareMathOperator{\Whit}{Whit}

\DeclareMathOperator{\Hom}{Hom}
\DeclareMathOperator{\Spec}{Spec}
\DeclareMathOperator{\rep}{-rep}
\DeclareMathOperator{\can}{can}
\DeclareMathOperator{\IndCoh}{IndCoh}

\newtheorem{theorem}{Theorem}[section]
\newtheorem{lemma}[theorem]{Lemma}
\newtheorem{proposition}[theorem]{Proposition}
\newtheorem{corollary}[theorem]{Corollary}
\newtheorem{definition}[theorem]{Definition}

\theoremstyle{remark}
\newtheorem{example}[theorem]{Example}
\newtheorem{remark}[theorem]{Remark}

\title{Categorical Moy-Prasad theory}
\author{David Yang}

\nocite{*}

\begin{document}

\maketitle

\begin{abstract}
We categorify the theory developed by Moy-Prasad in \cite{MP1}. More precisely, we define a depth filtration on any category with an action of the loop group $G((t))$ and prove a $2$-categorical generation statement inspired by the theory of unrefined minimal K-types. Using our generation theorem, we compute the depth filtration on the category of Whittaker sheaves and on the category of Kac-Moody modules. On the Whittaker side, we show that it recovers and extends the filtration constructed by Raskin in \cite{Whit}. For KM modules, our computation encodes several new localization statements, as well as confirming a conjecture of Chen-Kamgarpour.
\end{abstract}

\tableofcontents

\section{Introduction}

The local Langlands conjecture roughly predicts that representations of a group $G(F)$ over a local field can be parametrized in terms of representations of $\operatorname{Gal}(\overline{F}/F)$ into a Langlands dual group. Structures on one side of the correspondence should thus also exist on the other side.

On the Galois side, one salient piece of structure is the upper-numbered ramification filtration on $\operatorname{Gal}(\overline{F}/F)$. This is a descending filtration $\operatorname{Gal}(\overline{F}/F)^r$ indexed by rational numbers $r$, with $\operatorname{Gal}(\overline{F}/F)^{-1}$ the full Galois group and $\operatorname{Gal}(\overline{F}/F)^{0}$ the inertia subgroup. For any representation of $\operatorname{Gal}(\overline{F}/F)$, we define its depth to be the infimum of those $r$ for which $\operatorname{Gal}(\overline{F}/F))^r$ acts trivially. As an example, a representation is of depth $0$ exactly when it is tamely ramified.

We thus expect to be able to associate a corresponding invariant to an irreducible representation of $G(F)$, which would provide a measure of how wildly ramified it is. This was accomplished in \cite{MP1} by Moy-Prasad. In fact, they do more: they define the notion of an unrefined minimal K-type, which comes with an intrinsic depth. Each irreducible representation of $G(F)$ contains a unrefined minimal K-type (typically not unique), and Moy-Prasad show that its depth does not depend on the choice of unrefined minimal K-type. Thus, it defines an intrinsic invariant of the representation. In good cases (e.g., in sufficiently large characteristic), local Langlands should intertwine these two notions of depth.

Our goal is to categorify all the previously-mentioned mathematics. The context in which we work is that of the local geometric Langlands conjecture. Instead of working with a group $G(F)$, we work over an algebraically closed field of characteristic $0$ and consider the loop group $G((t))$. Instead of representations, we work with cocomplete dg-categories acted on by the monoidal category of D-modules $D(G((t))$ (at critical level, but we will not worry about the level in this introduction). The local geometric Langlands conjecture proposes that such categories should roughly correspond to categories over the stack of local systems $\operatorname{LocSys}(D^{\times}).$

A point of $\operatorname{LocSys}(D^{\times})$ corresponds to a D-module on the punctured disk, i.e., a possible singularity type. To such a singularity, one can associate a rational number, called its slope --- see for instance Section 2 of \cite{CK}. This is a close analogue to the notion of depth for a Galois representation. We thus have a sequence of substacks $\operatorname{LocSys}(D^{\times})_{\leq r}$ of local systems of slope at most $r$, and we can base change a category over $\operatorname{LocSys}(D^{\times})$ to these substacks. The result is a rational-indexed filtration on our original category. Applying Langlands, we see that any cocomplete dg-category $C$ with an action of $D(G((t)))$ should admit a natural such filtration $C_{\leq r}$.

In \cite{MP1}, Moy-Prasad construct a family of subgroups $K_{x,r+}$ of $G(F)$ and define the depth of an irrep $V$ to be the smallest $r$ for which some $V^{K_{x,r+}}$ is nontrivial. Previous work of Chen-Kamgarpour in \cite{CK} had defined the corresponding subgroups of $G((t))$, and used the same definition to define the depth of such a category $C$ as a rational number. However, as explained above, we actually expect the more refined data of a depth filtration on $C$.

To define the $C_{\leq r}$, we would like to take it to be the subcategory generated by the $V^{K_{x,r+}}$. This can be made sense of via a technical lemma proved by Ben-Zvi-Gunningham-Orem in \cite{BGO}. We take this approach in Section 2 to construct the depth filtration, and prove that it has some expected properties. 

In Section 3, we introduce our key technical tool for computing the depth filtration, Theorem \ref{generation}. As mentioned above, \cite{MP1} shows that every irrep of depth $r$ contains an unrefined minimal K-type of depth $r$. We reinterpret this as a generation statement: Every $G((t))$-category purely of depth $r$ admits a nontrivial map from at least one of a set of generators, approximately corresponding to the possible unrefined minimal K-types. In particular, to check that a representation is purely of depth $r$, it suffices to check that it admits no nontrivial maps from any generators of depth $r'\neq r$.

The remainder of the paper is dedicated to using Theorem \ref{generation} to compute the depth filtration on some categories of interest. First, in Section 4, we compute it on the category $\hat{\mathfrak{g}}_{\crit}\operatorname{-mod}$ of representations of the affine Lie algebra at critical level. This computation is closely related to the aforementioned work of Chen-Kamgarpour --- in particular, the results of \cite{CK} were what inspired us to conjecture Theorem \ref{depthKM}. They also implicitly give a proof of one direction of Theorem \ref{depthKM}, which we reproduce in our language.

As one application, the computation of $\hat{\mathfrak{g}}_{\crit}\operatorname{-mod}_{\leq r}$ implies (but is much stronger than) the main conjecture of \cite{CK}. Furthermore, it implies a number of new localization statements relating categories of Kac-Moody representations and categories of D-modules. For instance, in a future paper, joint with Sam Raskin, we plan to combine the methods of this paper with those of \cite{raskinFG} to prove a conjecture of Frenkel-Gaitsgory from \cite{FG2}.

Finally, in Section 5, we compute the depth filtration on the category $\Whit$ of Whittaker sheaves on $G((t))$. The answer recovers and extends the adolescent Whittaker construction of \cite{Whit}. Indeed, the original motivation behind this paper was to understand Remark 1.22.2 of loc.cit., which suggested a link between their construction and the slope filtration on $\operatorname{LocSys}(D^{\times}).$

\subsection{Notation}

Throughout this paper, we fix an algebraically closed field $k$ of characteristic $0$ and a connected reductive group $G$ over $k$. We also fix a Borel $B$ with unipotent radical $N$ and Cartan $T$. The Lie algebras of these groups will be denoted by the respective Fraktur letters, e.g., $\mathfrak{g}$ for $G$. The weight and coweight lattices will be denoted by $X^*(T)$ and $X_*(T)$, respectively.

We define the loop group $G((t))$ as the ind-scheme representing the functor $\Hom(\Spec R, G((t)))\cong\Hom(\Spec R((t)),G).$ Its Lie algebra may be identified with $\mathfrak{g}((t))$, the space of Laurent series with coefficients in $\mathfrak{g}$. The spherical subgroup of $G((t))$ will be denoted by $G[[t]]$, with Lie algebra $\mathfrak{g}[[t]]$.

All categories we consider are cocomplete $k$-linear stable $\infty$-categories (equivalently cocomplete dg-categories) and all functors we consider will be continuous (i.e., colimit preserving). We will often use the notion of the quotient of a morphism $C\rightarrow D$, not necessarily fully faithful. It is defined to be the cokernel (i.e., the pushout with $C\rightarrow 0$) in the category of cocomplete presentable categories. We write it as $\operatorname{quot} C\rightarrow D$.

Our main objects of study are categories $C$ with a (strong) action of $G((t))$ --- for this notion, as well as for the concepts of categorical representation theory, we refer to \cite{Beraldo}. An important special case is the (derived) category of $D$-modules on some stack $X$ acted on by $G((t))$, denoted $D(X)$. For the theory of $D$-modules on infinite dimensional stacks, we refer to \cite{dmodinf}.

More generally, we can (and should) choose some level $\kappa$ and consider the category of twisted D-modules $D_{\kappa}(G((t))$. Unless otherwise stated (e.g., in Section 4), all of our results will work equally well at all levels. However, we suppress the level in statements and proofs for ease of notation. Whenever we take invariants with respect to some subgroup of $G((t))$, this subgroup is pro-unipotent and the level canonically splits over it. This fact is sufficient (again, unless otherwise stated) to generalize our proofs to all levels.

\subsection{Acknowledgements}

This paper originated from an attempt to understand a remark by Sam Raskin in \cite{Whit}, and the form of our generation theorem was inspired by the proofs of \cite{raskinFG}. We thank him for his patient explanations of those papers and others. We thank Siyan Daniel Li-Huerta for telling us about the paper \cite{MP1}. This paper also benefitted from conversations with Lin Chen, Gurbir Dhillon, Charles Fu, James Tao, Jonathan Wang, and Yifei Zhao. Finally, we would like to thank our advisor, Dennis Gaitsgory. He in large part created this flavor of categorical representation theory, which has served as such a nice playground for us. His guidance has been truly invaluable. 

\section{Definition}

\subsection{Moy-Prasad subgroups}

We extensively use a series of subgroups of $G((t))$ first described (in the p-adic setting) by Moy and Prasad in \cite{MP1}.

Let $\mathfrak{g}=\oplus\mathfrak{g}_{\alpha}$ be the decomposition of $\mathfrak{g}$ into weight spaces. For every choice\footnote{It is more customary to take $x$ a point in the Bruhat-Tits building associated to $G$. We restrict to an apartment $X_*(T)\otimes\mathbb{R}$ here partly to make this paper more self-contained and partly because of the author's ignorance of Bruhat-Tits theory.} of a point $x$ in $X_*(T)\otimes\mathbb{R}$ and a nonnegative real number $r$, we can define two Lie subalgebras of $\mathfrak{g}((t))$ by

$$\mathfrak{k}_{x,r}=\displaystyle\bigoplus_{\langle\alpha,x\rangle+i\geq r}\mathfrak{g}_{\alpha}t^i$$
$$\mathfrak{k}_{x,r+}=\displaystyle\bigoplus_{\langle\alpha,x\rangle+i > r}\mathfrak{g}_{\alpha}t^i$$

\noindent which exponentiate to subgroups $K_{x,r}$ and $K_{x,r+}$ of $G((t))$. In the special case $r=0$, $K_{x,0}$ will be a parahoric subgroup and $K_{x,0+}$ its nilpotent radical. We will denote $K_{x,0}$ by $P_x$ and $K_{x,0}/K_{x,0+}$ by $L_x$. For any $r$, the inclusions

$$K_{x,r+}\subseteq K_{x,r}\subseteq P_x$$

\noindent are inclusions of normal subgroups. This gives us an action of $P_x$ on $K_{x,r}/K_{x,r+}$, and the reader may check that this action factors through $P_x\rightarrow L_x$.

\begin{example}\label{congruence}
Consider the case $x=0$ and $r$ an integer. Then $K_{x,r}$ is the $r$th congruence subgroup and $K_{x,r+}$ is the $(r+1)$-th. The group $P_x$ and $L_x$ can be identified with $G[[t]]$ and $G$.
\end{example}

\subsection{Depth filtration}

In the arithmetic setting, we can define the depth $\leq r$ part of a smooth representation $V$ of a p-adic group as the sub-representation generated by $K_{x,r+}$-invariant vectors, as $x$ ranges over all points of $X_*(T)\otimes\mathbb{R}$. Moving to the categorical setting, we need a replacement for the operation of taking the sub-representation generated by some specified elements. This is provided by the following lemma, which is a special case of Proposition 3.2 of Ben-Zvi-Gunningham-Orem \cite{BGO}:

\begin{lemma}
Let $C$ be a category acted on by $G((t))$. Then there is a natural $G((t))$-equivariant fully faithful embedding $$a_{x,r}:D(G((t))/K_{x,r+})\otimes_{D(K_{x,r+}\backslash G((t))/K_{x,r+})} C^{K_{x,r+}}\rightarrow C$$ with a continuous $G((t))$-equivariant right adjoint $a_{x,r}^R$.
\end{lemma}

We reproduce their proof here for the sake of self-containment. Throughout, we make liberal use of the foundational fact (proven e.g. in \cite{Beraldo}) that
$$D(G((t)))^K\cong D(G((t))/K)\cong D(G((t)))_K.$$

\begin{proof}
We treat the case of $C\cong D(G((t)))$. In this case, $a_{x,r}$ and $a_{x,r}^R$ will be adjoint functors in the category of $G((t))$-bimodules, so the general statement follows by tensoring $\otimes_{G((t))}M$.

To match with the notation of \cite{BGO}, we take $A=D(G((t)))$, $B=D(P_x/K_{x,r+})$, and $M=D(G((t))/K_{x,r+})$. As the notation suggests, $M$ is an $A-B$ bimodule. Let us rewrite the source of $a_{x,r}$ via a bar complex. Define a simplicial category $C_{\bullet}$ with

$$C_n\cong(M\otimes_B M)^{\otimes_A(n+1)},$$

\noindent so for instance we have $C_1\cong (M\otimes_B M)\otimes_A(M\otimes_B M).$

To define the morphisms, we need maps $f:M\otimes_B M\rightarrow A$ (which will give us the face maps) and $g:B\rightarrow M\otimes_A M$ (which will give us the degeneracy maps). These are best understood in terms of \cite{BGO}'s notion of a properly dualizable bimodule, but we will take a more concrete approach here for simplicity.

Note that we have an equivalence

$$M\otimes_B M\cong D(G((t))/K_{x,r+}\times_{P_x/K_{x,r+}} K_{x,r+}\backslash G((t)))$$

\noindent where the right hand side is given by a twisted product. The desired $f$ is a pull-push map along the correspondence

\begin{equation*}
\begin{tikzcd}
 & G((t))\times_{P_x}G((t)) \arrow[ld] \arrow[rd] & \\
G((t))/K_{x,r+}\times_{P_x/K_{x,r+}} K_{x,r+}\backslash G((t)) & & G((t))
\end{tikzcd}
\end{equation*}

Note that $f$ has a continuous right adjoint (as the left map is placid in the sense of \cite{dmodinf} and the right map is ind-proper), which is even a map of $A$-bimodules. Similarly, we have an equivalence

$$M\otimes_A M\cong D(K_{x,r+}\backslash G((t))/K_{x,r+})$$

\noindent and $g$ corresponds to pushforward along the closed immersion

$$K_{x,r+}\backslash P_x/K_{x,r+}\rightarrow K_{x,r+}\backslash G((t))/K_{x,r+},$$

\noindent using that $D(K_{x,r+}\backslash P_x/K_{x,r+})\cong D(P_x/K_{x,r+})$. As for $f$, we see that $g$ has a continuous $B\times B$-equivariant right adjoint.

Now that we have defined $C_{\bullet}$, let us link it to our original statement. By the theory of bar complexes, we have

$$D(G((t))/K_{x,r+})\otimes_{D(K_{x,r+}\backslash G((t))/K_{x,r+})} D(K_{x,r+}\backslash G((t)))\cong\operatorname{colim}C_{\bullet},$$

\noindent and $a_{x,r}$ can be identified with $\operatorname{colim}C_{\bullet}\rightarrow A.$ Note that we have a canonical augmentation $\widetilde{C}_{\bullet}$ with $C_{-1}\cong A$ and $C_0\rightarrow C_{-1}$ given by the map $f$. By taking right adjoints, we produce an augmented cosimplicial category $\widetilde{C}^{\bullet}$, which gives us our desired right adjoint $\operatorname{colim}C_{\bullet}\cong\operatorname{lim}C^{\bullet}\rightarrow A.$ We see by inspection that both $\operatorname{colim}C_{\bullet}\rightarrow A$ and $\operatorname{lim}C^{\bullet}\rightarrow A$ are continuous and have natural $A$-biequivariant structures.

It remains to show fully faithfulness. For this, we use the monadicity-type results established in Section 4.7.5 of \cite{HA}, which require checking the Beck-Chevalley conditions. In this case, the check can be done by hand, and so the lemma holds.
\end{proof}

\begin{remark}
We emphasize that the only information used about the subgroups $K_{x,r+}$ is that they are normal and of finite codimension inside a parahoric subgroup. On the other hand, we do not know if the statement stays true for any ``compact open'' (e.g., finite codimension in $G(O)$) subgroup, the issue being that there is no longer any clear reason for $f$ to admit a continuous right adjoint.
\end{remark}

The full subcategory defined by $a_{x,r}$ can be alternatively characterized as follows.

\begin{corollary}\label{minimal}
The image of $a_{x,r}$ is the minimal (up to equivalence) full subcategory of $C$ that both contains $C^{K_{x,r+}}$ and inherits a $G((t))$ action from $C$.
\end{corollary}
\begin{proof}
Let $D$ be another full subcategory of $C$ satisfying our conditions. By assumption, we have $C^{K_{x,r+}}\cong D^{K_{x,r+}}.$ So by the previous lemma, we have a fully faithful embedding

\begin{align*}
D(G((t))/K_{x,r+})\otimes_{D(K_{x,r+}\backslash G((t))/K_{x,r+})} C^{K_{x,r+}}&\cong\\
D(G((t))/K_{x,r+})\otimes_{D(K_{x,r+}\backslash G((t))/K_{x,r+})} D^{K_{x,r+}}&\rightarrow D
\end{align*}

\noindent which proves the desired inclusion.
\end{proof}

\begin{definition}\label{depthdefinition}
Let $C$ be a category acted on by $G((t))$. The depth filtration on $C$ is the $\mathbb{R}_{\geq 0}$-indexed filtration with $C_{\leq r}$ the smallest full (dg-)subcategory of $C$ containing the image of $a_{x,r}$ for every $x\in\mathfrak{h}$.
\end{definition}

Let us quickly show that the depth filtration is exhaustive, i.e., we have that $\operatorname{colim}{C_{\leq r}}\cong C.$ For all $r$, we have $C^{K_{0,r+}}\subseteq C_{\leq r}.$ Thus, we have $\operatorname{colim}{C^{K_{0,r+}}}\subseteq\operatorname{colim}{C_{\leq r}}\subseteq C,$ but we have an equivalence $\operatorname{colim}{C^{K_{0,r+}}}\cong C$, and so the filtration must be exhaustive.

The following proposition, which comes from the theory of optimal points described in Section 6.1 of \cite{MP1}, ensures that this is a good definition:

\begin{lemma}\label{rationality}
\begin{enumerate}
\item There is a discrete set of rational numbers, depending on $G$, so that the depth filtration on any category with $G((t))$ action can jump only at those rational numbers.
\item For any $r$ and $C$, $C_{\leq r}$ is generated by the images of $a_{x,r}$ for finitely many values of $x$.
\end{enumerate}
\end{lemma}

\begin{proof}

For each element $\lambda\in X_*(T)$, we have a point $t^{\lambda}$ of $G((t))$. The adjoint action of $t^{\lambda}$ on $\mathfrak{g}((t))$ satisfies

$$\operatorname{Ad}_{t^{\lambda}}(\mathfrak{g}_{\alpha}t^i)=\mathfrak{g}_{\alpha}t^{i+\langle\lambda,\alpha\rangle}.$$

From this formula, we see that $\operatorname{Ad}_{t^{\lambda}}(\mathfrak{k}_{x,r+})=\mathfrak{k}_{x-\lambda,r+}$ and $\operatorname{Ad}_{t^{\lambda}}(K_{x,r+})=K_{x-\lambda,r+}$. We thus have a commutative diagram

\begin{equation*}
\begin{tikzcd}
D(G((t))/K_{x,r+})\otimes_{D(K_{x,r+}\backslash G((t))/K_{x,r+})} C^{K_{x,r+}}\arrow[r, "a_{x,r}"] \arrow[d] & C \arrow[d, "\operatorname{Id}"] \\
D(G((t))/K_{x-\lambda,r+})\otimes_{D(K_{x-\lambda,r+}\backslash G((t))/K_{x-\lambda,r+})} C^{K_{x-\lambda,r+}}\arrow[r, "a_{x-\lambda,r}"] & C
\end{tikzcd}
\end{equation*}

with the left arrow given by right multiplication by $t^{-\lambda}$ on the first factor and left multiplication by $t^{\lambda}$ on the second factor. As it is an equivalence, the maps $a_{x,r}$ and $a_{x-\lambda,r}$ have the same essential image, so it suffices to consider $x$ in a fundamental domain $S$ for the translation action of $X_*(T)$ on $X_*(T)\otimes\mathbb{R}$, which we can choose to be compact.

For any real number $n$, if we impose an upper bound $r\leq n$ then the space of choices for $r$ is now also compact. We see that only finitely many subgroups $K_{x,r+}$ appear for $x\in S$ and $r\leq n$, which implies the second part of our lemma. To deduce the first part, it suffices to show that for any subgroup $K$, the smallest $r$ for which it appears as a $K_{x,r+}$ is rational. But this is clear, as the space of such choices of $x,r$ is cut out by linear inequalities with rational coefficients.

\end{proof}

\begin{remark}
In fact, the rational numbers that appear must have denominator dividing a fundamental degree of $G$. See Proposition \ref{divisibility}.
\end{remark}

\begin{corollary}
The category $C_{\leq r}$ has a natural $G((t))$ action so that the embedding
$$a_r:C_{\leq r}\rightarrow C$$
is $G((t))$ equivariant. Furthermore, $a_r$ has a continuous $G((t))$-equivariant right adjoint $a_r^R$.
\end{corollary}

\begin{proof}

By the previous lemma, we know that there are $G((t))$-categories $C_1,C_2,\cdots, C_k$, with fully faithful embeddings $a^i:C_i\subseteq C$, such that:

\begin{itemize}
\item The $C_i$ generate $C_{\leq r}$.
\item The $a^i$ are naturally $G((t))$-equivariant.
\item The $a^i$ have continuous right adjoints $a^{i,R}$ which are also naturally $G((t))$-equivariant.
\end{itemize}

Furthermore, we have

\begin{itemize}
\item The $a^{i,R}$ commute, i.e., the restriction of $a^{i,R}$ to $C_j$ lands in $C_i\cap C_j$. 
\end{itemize}

To see this, recall that each $C_i$ is constructed functorially in the $G((t))$-category $C$, and that the inclusion $C_j\rightarrow C$ is $G((t))$-equivariant.

We show that this is enough to conclude the desired properties of $C_{\leq r}$. For each nonempty subset $S$ of $\{1,2,\cdots,k\}$, define

$$C^S\cong\displaystyle\bigcap_{i\in S}C_i.$$

Denoting the inclusion $C^S\rightarrow C$ by $a^S$, we immediately see that there is a $G((t))$ action on $C^S$ making $a^S$ $G((t))$-equivariant. To see the same for the right adjoint $a^{S,R}$, note that it is a composition of the $a^{i,R}$ (using that they commute).

Let $K$ be the category of subsets of $\{1,2,\cdots,k\}$, with morphisms given by containment (so that the empty set is the final object.) We have a natural functor $K\rightarrow\operatorname{DGCat}$, given by sending nonempty subsets $S$ to $C^S$ and sending the empty set to $C_{\leq r}$. By definition, this is a colimit diagram --- this gives us the $G((t))$ structure on $C_{\leq r}$ and the equivariant structure on $a_r$.

To analyze $a_r^R$, consider the functor $K^{\operatorname{op}}\rightarrow\operatorname{DGCat}$ obtained by taking right adjoints. By Corollary 5.5.3.4 of \cite{HTT}, this is a limit diagram, i.e., we have an equivalence

$$C_{\leq r}\cong\displaystyle\lim_S C^S,$$

\noindent which gives us the desired statements on $a_r^R$.
\end{proof}

\subsection{Some auxiliary lemmas}

We now prove a number of basic statements which will be used later. We start with an orthogonality statement.

\begin{lemma}\label{orthogonal}
Let $C$ and $D$ be $G((t))$-categories. Assume that $C\cong C_{\leq r}$ and $D_{\leq r}\cong 0.$ Then,

$$C\otimes_{G((t))}D\cong 0.$$
\end{lemma}

\begin{proof}
Take $C_i$ as in the previous proofs. As $C$ can be written as a colimit of the $C^S$, we are reduced to showing the result for $C$ replaced with $C^S$, which in turns reduces to the case of $C\cong C_i$. So WLOG, we can assume $C\cong D(G((t))/K_{x,r+})\otimes_{D(K_{x,r+}\backslash G((t))/K_{x,r+})} C^{K_{x,r+}}$. But then we have

\begin{align*}
C\otimes_{G((t))}D&\cong (C^{K_{x,r+}}\otimes_{D(K_{x,r+}\backslash G((t))/K_{x,r+})} D(G((t))/K_{x,r+}))\otimes_{G((t))} D\\
&\cong C^{K_{x,r+}}\otimes_{D(K_{x,r+}\backslash G((t))/K_{x,r+})}D^{K_{x,r+}}\\
&\cong C^{K_{x,r+}}\otimes_{D(K_{x,r+}\backslash G((t))/K_{x,r+})}0\\
&\cong 0,
\end{align*}

\noindent as desired.
\end{proof}

Next, we have the following lemma which will be useful later for induction arguments.

\begin{lemma}\label{inductionaux}
Let $C$ be a $G((t))$-category. Denote by $b_r: C\rightarrow C_{=r}$ the quotient of the map $a_r:C_{< r}\rightarrow C_{\leq r}.$ Then $b_r^R$ is continuous and fully faithful and $C_{\leq r}$ is generated by the images of $a_r$ and $b_r^R$.
\end{lemma}

\begin{proof}
By 5.5.3.4 of \cite{HTT}, the cokernel (in the category of presentable categories) of $a_r$ is equivalent to the kernel of $a_r^R$, and under this equivalence the map $b_r^R$ is sent to the structure map of the kernel, which shows that it is continuous and fully faithful.

To see the generation statement, let $X$ be an object of $C$. Denote the cone of the counit map $a_ra_r^RX\rightarrow X$ by $X'.$ Now note that $a_r^RX'$ is the cone of $a_r^Ra_ra_r^RX\cong a_r^RX\rightarrow a_r^RX,$ and hence trivial. Therefore $X'$ lies in the image of $b_r^R$, and in fact $X'\cong b_r^Rb_rX'\cong b_r^Rb_rX$, so we have an exact triangle

$$a_ra_r^RX\rightarrow X\rightarrow b_r^Rb_r X,$$

\noindent which shows the generation statement.
\end{proof}

Finally, we prove a statement about the stability of depth under subquotients.

\begin{lemma}\label{subquotient}
Let $f:C\rightarrow D$ be a map of $G((t))$-categories.
\begin{enumerate}
\item If $D\cong D_{\leq r}$ and $f$ is conservative, then $C\cong C_{\leq r}$.
\item If $C\cong C_{\leq r}$ and the essential image of $f$ generates $D$ under colimits, then $D\cong D_{\leq r}$.
\end{enumerate}
\end{lemma}

\begin{proof}
First assume that $D\cong D_{\leq r}$ and $f$ is conservative. We have a commutative diagram

\begin{equation*}
\begin{tikzcd}
C_{\leq r} \arrow[d, "f_{\leq r}"] & C \arrow[d, "f"] \arrow[l,"a_r^R"] \\
D_{\leq r} & D \arrow[l, "a_r^R"].
\end{tikzcd}
\end{equation*}

By our assumptions, the composition $C\rightarrow D\rightarrow D_{\leq r}$ is conservative, so the composition $C\rightarrow C_{\leq r}\rightarrow D_{\leq r}$ is conservative. This implies that $C\rightarrow C_{\leq r}$ is conservative, which, combined with fully faithfulness of $a_r$, shows that $C\cong C_{\leq r}$.

Assume on the other hand that $C\cong C_{\leq r}$ and that the essential image of $f$ generates $D$ under colimits. This time we use the commutative diagram

\begin{equation*}
\begin{tikzcd}
C_{\leq r} \arrow[d, "f_{\leq r}"] \arrow[r, "a_r"] & C \arrow[d, "f"] \\
D_{\leq r} \arrow[r, "a_r"] & D.
\end{tikzcd}
\end{equation*}
\end{proof}

Similar logic to the previous case shows that the image of $D_{\leq r}\rightarrow D$ must generate $D$ under colimits, but the properties of $a_r$ imply that we must then have $D\cong D_{\leq r}$.

\section{Moy-Prasad generators}

\subsection{Statement of the main theorem}

Let $G((t))\operatorname{-rep}$ be the category of categories with a strong $G((t))$ action. This is naturally an $(\infty,2)$-category, in fact a category enriched in dg-categories. The notion of an enriched ${\infty}$-category is developed in \cite{enriched} --- however, for the purposes of this paper, all our examples will be built from the $\infty$-category of dg-categories and can be dealt with in an ad hoc way.

We define $G((t))\operatorname{-rep}_{\leq r}$ and $G((t))\operatorname{-rep}_{\geq r}$ to be the full subcategories of $G((t))\operatorname{-rep}$ defined by $C_{\leq r}\cong C$ and $C_{< r}\cong 0$, respectively. Their intersection will be denoted $G((t))\operatorname{-rep}_{=r}$. As the constructions of the previous section are functorial, we have truncation functors $\tau_{\leq r}:G((t))\rep\rightarrow G((t))\rep_{\leq r}$ and $\tau_{<r}:G((t))\rep\rightarrow G((t))\rep_{<r}.$ Taking quotients, we also produce truncation functors $\tau_{>r}:G((t))\rep\rightarrow G((t))\rep_{>r}$ and $\tau_{\geq r}:G((t))\rep\rightarrow G((t))\rep_{\geq r}.$ Our goal is to produce a set of generators of $G((t))\rep_{=r}$ in the following sense:

\begin{definition}
Let $C$ be a category enriched over the category of dg-categories. Say that a set $S$ of objects of $C$  generates $C$ if, for any non-final object $c\in C$, there is some $s\in S$ such that $\operatorname{Hom}(s,c)$ is not the trivial dg-category.
\end{definition}

The key computation is that of $\tau_{\geq r}D(G((t))/K_{x,r+}).$ We shall assume $r>0$ for the rest of the section, as the other case is trivial. Let us introduce some notation to state the result.

As $K_{x,r+}$ is normal inside $K_{x,r}$, we have a right action of $K_{x,r}/K_{x,r+}$ on $D(G((t))/K_{x,r+})$. Because $r>0$, then the group $K_{x,r}/K_{x,r+}$ is in fact naturally isomorphic to the additive group on a vector space. The Fourier transform of \cite{Beraldo} then identifies the convolution monoidal structure on $D(K_{x,r}/K_{x,r+})$ with the $!$-tensor monoidal structure on $D((K_{x,r}/K_{x,r+})^*)$.

As briefly mentioned in the previous section, we have an action of $L_x$ on $K_{x,r}/K_{x,r+}$, and hence on $(K_{x,r}/K_{x,r+})^*$. Let $(K_{x,r}/K_{x,r+})^{*,\circ}$ denote the locus of GIT-semistable \cite{GIT} elements, i.e., elements whose $L_x$ orbit does not contain zero in its closure. Then we can define 

$$D(G((t))/K_{x,r+})^{\circ}\cong D(G((t))/K_{x,r+})\otimes_{D((K_{x,r}/K_{x,r+})^*)}D((K_{x,r}/K_{x,r+})^{*,\circ}).$$

The main result of this section is the following. It is a categorification of the theory of unrefined minimal $K$-types developed in \cite{MP1}, and its proof is based on the same ideas.

\begin{theorem}\label{generation}
$$\tau_{\geq r}D(G((t))/K_{x,r+})\cong D(G((t))/K_{x,r+})^{\circ}.$$
\end{theorem}

\begin{example}
Continuing Example \ref{congruence}, consider the case of $x=0$ and $r$ an integer. The space $K_{x,r}/K_{x,r+}$ can be identified with the Lie algebra $\mathfrak{g}$ of $G$, and we are interested in the semistable locus of $\mathfrak{g}^*$ under the coadjoint action. Using the Chevalley isomorphism, this can be computed to be the complement of the nilpotent cone.
\end{example}

Let us explain how to derive a generation statement from Theorem \ref{generation}. As $D((K_{x,r}/K_{x,r+})^{*,\circ})$ is self-dual as a $D((K_{x,r}/K_{x,r+})^*)$-module, we have

\[\Hom_{G((t))}(D(G((t))/K_{x,r+})^{\circ},C)\cong D(G((t))/K_{x,r+})^{\circ}\otimes_{G((t))}C.\]

To simplify our notation, we will denote these equivalent categories by $C^{K_{x,r+},\circ}.$

\begin{corollary}
Let $C$ be a nontrivial category in $G((t))\operatorname{-rep}_{=r}.$ Then there is some $x\in X_*(T)\otimes\mathbb{R}$ with $C^{K_{x,r+},\circ}\not\cong 0$. Equivalently, the $D(G((t))/K_{x,r+})^{\circ}$ form a set of generators for $G((t))\operatorname{-rep}_{=r}.$
\end{corollary}

\begin{proof}
As $C_{\leq r}\cong C$ is nontrivial, there must be some $x$ such that $C^{K_{x,r+}}\not\cong 0$. Now note that

\begin{align*}
C^{K_{x,r+},\circ}&\cong C\otimes_{G((t))}D(G((t))/K_{x,r+})^{\circ}\\
&\cong C\otimes_{G((t))}(\operatorname{quot}\tau_{<n}D(G((t))/K_{x,r+})\rightarrow D(G((t))/K_{x,r+}))\\
&\cong\operatorname{quot} C\otimes_{G((t))}\tau_{<n}D(G((t))/K_{x,r+})\rightarrow C^{K_{x,r+}}\\
&\cong\operatorname{quot} 0\rightarrow C^{K_{x,r+}}\\
&\cong C^{K_{x,r+}}\not\cong 0,
\end{align*}

\noindent where we used Lemma \ref{orthogonal} in the second to last line.
\end{proof}

Ranging over all $r$, we get:

\begin{corollary}
The $D(G((t))/K_{x,r+})^{\circ}$, as $(x,r)$ ranges over all pairs in $(X_*(T)\otimes\mathbb{R})\times\mathbb{R}_{>0},$ form a system of generators of $G((t))\operatorname{-rep}$.
\end{corollary}

\begin{proof}
Assume $C$ is a $G((t))$-category with $C^{K_{x,r+},\circ}\cong 0$ for all choices of $x,r.$ Let us show that all the $C_{\leq r}$ are trivial. Assume otherwise --- then there exists some $r$ so that $C_{<r}\cong 0$ but $C_{\leq r}\not\cong 0$. But then $C_{\leq r}\in G((t))\operatorname{-rep}_{=r},$ so we have a contradiction by the previous corollary. By the exhaustiveness of the depth filtration (see the discussion after Definition \ref{depthdefinition}), this implies that $C$ is trivial, as desired.
\end{proof}

The rest of this section will be devoted to the proof of Theorem \ref{generation}.
\subsection{Proof of Theorem \ref{generation}}
\subsubsection{Outline}
To understand the proof, it is helpful to consider the following case. Let $y$ be another element of $X_*(T)\otimes\mathbb{R}$ and consider a single $K_{y,r}\times K_{x,r}$ orbit $K_{y,r}gK_{x,r}/K_{x,r+}$ inside $G((t))/K_{x,r+}$. By the definition of the depth filtration, any $K_{y,r}$-equivariant $D$-module supported on this orbit lies inside $\tau_{<r}D(G((t))/K_{x,r+})$. Hence the theorem predicts that, as a $D((K_{x,r+}/K_{x,r+})^*)$-module, $D(K_{y,r}\backslash K_{y,r}gK_{x,r}/K_{x,r+})$ should be supported on the unstable locus.

To explain this phenomenon, note that the right $K_{x,r}/K_{x,r+}$ action on $K_{y,r}\backslash K_{y,r}gK_{x,r}/K_{x,r+}$ has nontrivial kernel (and so is not faithful), given by $((g^{-1}K_{y,r}g\cap K_{x,r})+K_{x,r+})/K_{x,r+}.$ This automatically implies that the support of $D(K_{y,r}\backslash K_{y,r}gK_{x,r}/K_{x,r+})$ as a $D((K_{x,r+}/K_{x,r+})^*)$-module lies in $(g^{-1}K_{y,r}g\cap K_{x,r})^{\perp},$ and this turns out to always consist only of unstable points.

To see this, we apply the Bruhat decomposition to write $g$ as $p_ywp_x$, where $p_x\in P_x, p_y\in P_y,$ and $w$ is an element of the affine Weyl group. Then $g^{-1}K_{y,r}g=p_x^{-1}K_{wy,r}p_x,$ so $(g^{-1}K_{y,r}g\cap K_{x,r})^{\perp}\cong(\operatorname{Ad}p_x^{-1}(K_{wy,r}\cap K_{x,r}))^{\perp}.$ By Lemma \ref{input} below, this lies inside the unstable locus.

The proof below will be a bit more complicated, in order to run this argument in families, i.e., to consider more than one double coset at a time.

\begin{proof}

Denoting the unstable locus of $(K_{x,r}/K_{x,r+})^*$ by $(K_{x,r}/K_{x,r+})^*_{us},$ we can define

$$D(G((t))/K_{x,r+})_{us}\cong D(G((t))/K_{x,r+})\otimes_{D((K_{x,r}/K_{x,r+})^*)}D((K_{x,r}/K_{x,r+})^*_{us}).$$

So our theorem is equivalent to the statement that $\tau_{< r}D(G((t))/K_{x,r+})$ and $D(G((t))/K_{x,r+})_{us}$ are equal as full subcategories of $D(G((t))/K_{x,r+})$. We prove this by showing inclusions both ways. The following lemma, whose proof we delay, isolates the key geometric input. It is essentially a version of the propositions in Section 6 of \cite{MP1}.

\begin{lemma}\label{input}
Let $Z,Z'\subseteq(K_{x,r}/K_{x,r+})^*$ be defined by

$$Z=\displaystyle\bigcup_{y}(K_{x,r}\cap K_{y,r})^{\perp}$$

\noindent and

$$Z'=\displaystyle\bigcup_{y\mid K_{x,r+}\subseteq K_{y,r}\subseteq K_{x,r}}K_{y,r}^{\perp}.$$

Then these unions can be taken to be finite and $L_x\cdot Z=L_x\cdot Z'=(K_{x,r}/K_{x,r+})^*_{us}$.
\end{lemma}

\subsubsection{The first inclusion}

First we show that $\tau_{< r}D(G((t))/K_{x,r+})\subseteq D(G((t))/K_{x,r+})_{us}$. By Corollary \ref{minimal}, it suffices to show that, for any $y\in X_*(T)\otimes\mathbb{R}$ and $s < r$, we have $D(G((t))/K_{x,r+})^{K_{y,s+}}\subseteq D(G((t))/K_{x,r+})_{us}.$ As $K_{y,r}\subseteq K_{y,s+}$, it suffices to show that $D(G((t))/K_{x,r+})^{K_{y,r}}\subseteq D(G((t))/K_{x,r+})_{us}.$ Equivalently, we need to show that the induced action of $D((K_{x,r}/K_{x,r+})^{*,\circ})$ on $D(K_{y,r}\backslash G((t))/K_{x,r+})$ is trivial.

The Bruhat decomposition tells us that we can write $G((t))$ as a union of strata $P_ywP_x$, where $w$ is an element of the extended affine Weyl group. Thus, by a d\'evissage argument, we are reduced to showing that $D((K_{x,r}/K_{x,r+})^{*,\circ})$ acts trivially on $D(K_{y,r}\backslash P_ywP_x/K_{x,r+})$.

We have a smooth cover
$$(K_{y,r}\backslash P_y)\times((P_x\cap w^{-1}K_{y,r}w)\backslash P_x/K_{x,r+})\rightarrow K_{y,r}\backslash P_ywP_x/K_{x,r+}$$

\noindent given by $(a,b)\mapsto awb$. As pullback along this map is conservative, it suffices to show that $D((K_{x,r}/K_{x,r+})^{*,\circ})$ acts trivially on $D((K_{y,r}\backslash P_y)\times((P_x\cap w^{-1}K_{y,r}w)\backslash P_x/K_{x,r+})).$ Here the action is entirely on the second factor, so by Corollary 8.3.4 of \cite{DrGa}, it is enough to prove the triviality of the action on $D((P_x\cap w^{-1}K_{y,r}w)\backslash P_x/K_{x,r+})$. Once again, we can pass to a smooth cover and work instead with $D((K_{x,r}\cap w^{-1}K_{y,r}w)\backslash P_x/K_{x,r+}).$

Using that $r>0$, we have a natural $K_{x,r}/K_{x,r+}$-equivariant map

$$(K_{x,r}\cap w^{-1}K_{y,r}w)\backslash P_x/K_{x,r+}\rightarrow L_x,$$

\noindent where the action on $L_x$ is trivial.

The fiber over an element $l\in L_x$ is isomorphic to the stack $(K_{x,r}\cap w^{-1}K_{y,r}w)\backslash l\cdot K_{x,0+}/K_{x,r+}.$ This is the relative classifying stack of a flat unipotent group over the variety $K_{x,0+}/(K_{x,r+}\cdot\operatorname{Ad}_{l^{-1}}(K_{x,r}\cap w^{-1}K_{y,r}w)),$ and hence we have a $K_{x,r}/K_{x,r+}$-equivariant equivalence

$$D((K_{x,r}\cap w^{-1}K_{y,r}w)\backslash l\cdot K_{x,0+}/K_{x,r+})\cong D(K_{x,0+}/(K_{x,r+}\cdot\operatorname{Ad}_{l^{-1}}(K_{x,r}\cap w^{-1}K_{y,r}w))).$$

The action of $D(K_{x,r}/K_{x,r+})$ on $D(K_{x,0+}/(K_{x,r+}\times\operatorname{Ad}_{l^{-1}}(K_{x,r}\cap w^{-1}K_{y,r}w)))$ factors through 

$$D(K_{x,r}/(K_{x,r+}\cdot\operatorname{Ad}_{l^{-1}}(K_{x,r}\cap w^{-1}K_{y,r}w)))\cong D(\operatorname{Ad}_{l^{-1}}(K_{x,r}\cap w^{-1}K_{y,r}w)^{\perp}),$$

\noindent where $\operatorname{Ad}_{l^{-1}}(K_{x,r}\cap w^{-1}K_{y,r}w)^{\perp}$ is considered as a subspace of $(K_{x,r}/K_{x,r+})^*$. So we need to show that this subspace lies in $(K_{x,r}/K_{x,r+})_{us}.$ This follows from the lemma and the observation that $w^{-1}K_{y,r}w\cong K_{w\cdot y,r}.$

\subsubsection{The second inclusion}

Now we prove the other direction. It suffices to show that $D(G((t))/K_{x,r+})_{us}$ is of depth strictly less than $r$. Our proof will in fact show that this category is generated by the images of the ${a_{y,s}}$, where the pair $(y,s)$ ranges over those choices with $K_{x,r+}\subseteq K_{y,r}\subseteq K_{x,r}$ and $s < r$. 

Choose a representative $y$ for each possible subgroup $K_{y,r}$ satisfying $K_{x,r+}\subseteq K_{y,r}\subseteq K_{x,r}$, along with a $s$ such that $K_{y,s+}\cong K_{y,r}$. We have a natural convolution map

$$\displaystyle\bigcup_y D(G((t))/K_{y,r})\otimes D(K_{y,r}\backslash P_x)\rightarrow D(G((t))/K_{x,r+})$$

\noindent (here we use that $D(K_{y,r}\backslash P_x)\cong D(K_{y,r}\backslash P_x/K_{x,r+}).$) The left hand side is clearly of depth $< r$, and so this map must land in $D(G((t))/K_{x,r+})_{us}$. If we show that the image of this map generates $D(G((t))/K_{x,r+})_{us}$ under colimits, we will done by Lemma \ref{subquotient}. Let us understand this map in the Fourier-transformed picture. The category $D(G((t))/K_{y,r})$ can be rewritten:

\begin{align*}
D(G((t))/K_{y,r})&\cong D(G((t))/K_{x,r+})\otimes_{D(K_{y,r}/K_{x,r+})}\operatorname{Vect}\\
&\cong D(G((t))/K_{x,r+})\otimes_{D(K_{x,r}/K_{x,r+})}D(K_{x,r}/K_{y,r})\\
&\cong D(G((t))/K_{x,r+})\otimes_{D((K_{x,r}/K_{x,r+})^*)}D(K_{y,r}^{\perp}).
\end{align*}

After applying an automorphism to the LHS, we can identify the previously mentioned convolution map with $D(G((t))/K_{x,r+})\otimes_{D((K_{x,r}/K_{x,r+})^*)}\pi_*$, where $\pi$ is the map

$$\displaystyle\bigcup_y(K_{y,r}^{\perp}\times (K_{y,r}\backslash P_x))\rightarrow\displaystyle\bigcup_y(K_{y,r}^{\perp}\times L_x)\rightarrow(K_{x,r}/K_{x,r+})^*_{us},$$

\noindent with the right map coming from the $L_x$-action. By the lemma, $\pi$ is surjective, so Lemma \ref{appendixgen} implies that the image of pushforward along $\pi$ generates $D((K_{x,r}/K_{x,r+})^*_{us})$ under colimits. The property of having generating image is stable under tensor product (it is equivalent to having trivial quotient) and so the convolution map retains this property, as desired.

\subsubsection{Proof of Lemma \ref{input}}

Now we explain the proof of the lemma.
\begin{proof}
First, note that $K_{y,r}$ is a sum of Kac-Moody weight spaces, and so there are only finitely many possibilities for the group $(K_{x,r}\cap K_{y,r})+K_{x,r+}.$ This immediately implies the desired finiteness of the unions, so we focus on the remaining statement.

We clearly have $L_x\cdot Z'\subseteq L_x\cdot Z$, so it suffices to prove that $(K_{x,r}/K_{x,r+})^*_{us}\subseteq L_x\cdot Z'$ and $L_x\cdot Z\subseteq(K_{x,r}/K_{x,r+})^*_{us}.$ We first show the latter statement.

As $(K_{x,r}/K_{x,r+})^*_{us}$ is a $L_x$-invariant subvariety, it suffices to show that $Z\subseteq (K_{x,r}/K_{x,r+})^*_{us}$, or equivalently, that $(K_{y,r}\cap K_{x,r})^{\perp}\subseteq(K_{x,r}/K_{x,r+})^*_{us}$ for any $y\in X_*(T)\otimes\mathbb{R}$.

If we have the decomposition

$$K_{x,r}/K_{x,r+}\cong\displaystyle\bigoplus_{\alpha}\mathfrak{g}_{\alpha}t^{i_{\alpha}}$$

\noindent then we can identify

$$(K_{y,r}\cap K_{x,r}+K_{x,r+})/K_{x,r+}\cong\displaystyle\bigoplus_{\langle y-x,\alpha\rangle\geq 0}\mathfrak{g}_{\alpha}t^{i_{\alpha}}$$

\noindent and

$$(K_{y,r}\cap K_{x,r})^{\perp}\cong\displaystyle\bigoplus_{\langle y-x,\alpha\rangle > 0}\mathfrak{g}_{-\alpha}^*t^{i_{\alpha}}.$$

By the same logic as in Lemma \ref{rationality}, we can assume $y-x\in X_*(T)\otimes\mathbb{Q}$. Taking a multiple of $x-y$ which lies in $X_*(T)$, we find a one-parameter subgroup of $T$ which contracts $(K_{y,r}\cap K_{x,r})^{\perp}$ to the origin. This implies that $(K_{y,r}\cap K_{x,r})^{\perp}$ lies inside the unstable locus, as desired.

Now we prove that $(K_{x,r}/K_{x,r+})^*_{us}\subseteq L_x\cdot Z'$. We essentially reverse the logic of the previous implication. Let $p$ be a point of $(K_{x,r}/K_{x,r+})^*_{us}$. By the Hilbert-Mumford criterion, we have a one-parameter subgroup $\mathbb{G}_m\rightarrow L_x$ which contracts $p$ to the origin. We can conjugate this one-parameter subgroup to lie in $T$, so there is some $\beta\in X_*(T)$ and $l\in L_x$ such that

$$l\cdot p\in \displaystyle\bigoplus_{\langle \beta,\alpha\rangle > 0}\mathfrak{g}_{-\alpha}^*t^{i_{\alpha}}.$$

Then for sufficiently small $\epsilon > 0$, if we take $y=x+\epsilon\beta,$ we have $K_{x,r+}\subseteq K_{y,r}\subseteq K_{x,r}$ and

$$l\cdot p\in K_{y,r}^{\perp}$$

\noindent which implies the desired statement. (This argument also implies that $Z=Z'$, but we do not use this.)
\end{proof}
\end{proof}

\section{Depth filtration on KM}

\subsection{Review}

\subsubsection{The category $\hat{\mathfrak{g}}_{\crit}\operatorname{-mod}$}

We now analyze the depth filtration on an important $G((t))$-category. This category, which we will denote $\hat{\mathfrak{g}}\operatorname{-mod}_{\kappa}$, was defined in Section 23 of \cite{FG6} as a ``renormalized'' version of the category of representations of the affine lie algebra $\hat{\mathfrak{g}}$ at level $\kappa$. The renormalization is necessary for the existence of a $G((t))$-action (of level $\kappa$).

For a reader unfamiliar with this renormalization procedure, we advise ignoring it on a first pass through this section.  To understand the $G((t))$ action, it may help to know that for $K\subseteq G((t))$, the invariant category $\hat{\mathfrak{g}}\operatorname{-mod}_{\kappa}^K$ is a renormalized derived category of $K$-integrable $\hat{\mathfrak{g}}$-modules at level $\kappa$. On the other hand, for a detailed modern treatment of both the renormalization and the $G((t))$-action, we refer to \cite{semiinf}.

In this section, we restrict to the case of critical level, which we will shorthand with a subscript, e.g., $\hat{\mathfrak{g}}_{\crit}$. At all other levels, the universal enveloping algebra $U(\hat{\mathfrak{g}})_{\kappa}$ has trivial center. But at critical level, we have an identification between the center of $U(\hat{\mathfrak{g}})_{\crit}$ and the ring of functions on the space $\Op_{\check{\mathfrak{g}}}(\mathring{D})$ of $\check{\mathfrak{g}}$-opers on the punctured disk, where $\check{\mathfrak{g}}$ is the Langlands dual Lie algebra of $\mathfrak{g}.$ (Henceforth, we will denote this space by $\Op_{\check{\mathfrak{g}}}$, omitting the $\mathring{D}$. As we never consider opers on any space other than the punctured disk, we believe that this will cause no confusion.)

Let us quickly explain this statement (first proven in \cite{FF}). First we introduce opers (following \cite{FG2}.) For brevity's sake, we work explicitly in the coordinate $t$.

\subsubsection{Opers}

Let $\check{\rho}$ denote the half-sum of positive coroots of $\check{\mathfrak{g}}$. The adjoint action of $\check{\rho}$ gives us a $\mathbb{Z}$-grading $\check{\mathfrak{g}}\cong\oplus\check{\mathfrak{g}}_i,$ with $\check{\mathfrak{b}}\cong\check{\mathfrak{g}}_0\oplus\check{\mathfrak{g}}_1\oplus\cdots$. Pick $p_{-1}\in\check{\mathfrak{g}}_{-1}$ and $p_1\in\check{\mathfrak{g}}_1$ such that $\{p_{-1},2\check{\rho},p_1\}$ is a $\mathfrak{sl}_2$ triple. Then an oper on the punctured disk is a $N((t))$-gauge equivalence class of connections of the form $\nabla + (p_{-1}+\check{\mathfrak{b}}((t)))dt$, where $\nabla$ is the trivial connection (on the trivial bundle).

The following lemma, which appears as Proposition 1.3.1 in \cite{FG2} and was first proven in \cite{DS}, will be useful for working with opers. Let the kernel of $\operatorname{ad} p_1$ be denoted $V_{\can}=V_{\can,0}\oplus V_{\can,1}\oplus\cdots\subseteq\check{\mathfrak{b}},$ where we have $V_{\can,i}\subseteq \check{\mathfrak{g}}_i$.

\begin{lemma}
Every $N((t))$-gauge equivalence class as above has a unique representative of the form

$$\nabla+(p_{-1}+v(t))dt$$

\noindent with $v(t)\in V_{\can}((t))$.
\end{lemma}

So the moduli space of opers, defined by an a priori potentially nasty quotient, can in fact be represented (non-canonically) by an ind-pro-affine space. Let us rewrite this in terms of its ring of functions. As $\Op_{\check{\mathfrak{g}}}$ is an ind-scheme, its ring of functions is best represented by a topological algebra (see, for instance, Section 19.2 of \cite{FG2}.)

Note that any element $f\in V_{\can}^*((t))$ defines a function on $\Op_{\check{\mathfrak{g}}}$, namely by sending an oper $\nabla+(p_{-1}+v(t))dt$ to $\operatorname{Res}(\langle f,v(t)\rangle dt).$ This extends to an isomorphism of topological algebras

\begin{equation}\label{eq1}
\mathcal{O}(\Op_{\check{\mathfrak{g}}})\cong\operatorname{Sym}(V_{\can}^*((t))),
\end{equation}

\noindent where the topology on the symmetric algebra is that induced by the $\overset{!}{\otimes}$ tensor product (see, again, Section 19.1 of \cite{FG2}) from the Laurent series topology on $V_{\can}^*((t))$. We emphasize again that this isomorphism is not canonical\footnote{However, the structure of torsor for $V_{\can}((t))$ (with a modified $\operatorname{Aut}(\mathring{D})$-action) on $\Op_{\check{\mathfrak{g}}}$ is canonical, and this suffices to show the canonicity of our later constructions. See equation (4.2-3) of \cite{frenkelbook}.}: it is not equivariant with respect to $\operatorname{Aut}(\mathring{D})$.

\subsubsection{Feigin-Frenkel}

The relevance of opers for us comes from the previously mentioned Feigin-Frenkel isomorphism,

$$Z(U(\hat{\mathfrak{g}})_{\crit})\cong\mathcal{O}(\Op_{\check{\mathfrak{g}}}),$$

\noindent where $Z(U(\hat{\mathfrak{g}})_{\crit})$ denotes the center of $U(\hat{\mathfrak{g}})_{\crit},$ with the subspace topology. We now state two compatibilities that this isomorphism satisfies. Both the isomorphism and its compatibilities are proven in Theorem 4.3.2 of \cite{frenkelbook} and the discussion afterwards.

The first compatibility that we need is that the Feigin-Frenkel isomorphism exchanges the natural $\operatorname{Aut}(\mathring{D})$ actions on both algebras. In particular, it respects the conformal grading corresponding to the subgroup $\mathbb{G}_m\subseteq\operatorname{Aut}(\mathring{D}).$ We will denote the $i$th graded part by a superscript; for instance, using equation (4.2-3) of \cite{frenkelbook} we compute

\begin{equation}\label{grading}
V_{\can}^*((t))^i\cong\underset{n}{\oplus} V_{\can,n}^*t^{i+n}.
\end{equation}

To phrase the second compatibility, let us introduce natural filtrations on both sides. On the affine algebra side, we have the PBW filtration on $U(\hat{\mathfrak{g}}_{\crit})$, which induces a filtration by restriction on its center. For the other side, we use the description provided by \eqref{eq1}. We can define a filtration $F_iV_{\can}^*((t))$ by

$$F_iV_{\can}^*((t))=\underset{n<i}{\oplus}V_{\can,n}^*((t)),$$

\noindent and this induces a multiplicative filtration on $\operatorname{Sym}(V_{\can}^*((t)))\cong\mathcal{O}(\Op_{\check{\mathfrak{g}}}).$ So for instance, $F_2\operatorname{Sym}(V_{\can}^*((t)))\cong k\oplus V_{\can,0}^*((t))\oplus V_{\can,1}^*((t))\oplus\operatorname{Sym}^2(V_{\can,0}^*((t))).$

Although the isomorphism \eqref{eq1} is not canonical, this filtration is (and in particular, it is $\operatorname{Aut}(\mathring{D})$-equivariant and so is compatible with the conformal grading.) In fact, it is sent to the PBW filtration by Feigin-Frenkel. Furthermore, the map on associated graded algebras

\begin{equation}\label{compatibility}
\operatorname{gr}^F\mathcal{O}(\Op_{\check{\mathfrak{g}}})\cong\operatorname{Sym}(V_{\can}^*((t)))\rightarrow\operatorname{gr}^FU(\hat{\mathfrak{g}})_{\crit}\cong\operatorname{Sym}(\mathfrak{g}((t)))\end{equation}

\noindent can be described explicitly as follows. The LHS and RHS can be described as the ring of functions on the space of $k((t))$-points of $V_{\can}$ and $\mathfrak{g}^*$, respectively, so it suffices to produce a map of varieties $\mathfrak{g}^*\rightarrow V_{\can}$. By the theory of the Kostant section, we have a canonical isomorphism $V_{\can}\cong\check{\mathfrak{g}}//\check{G}\cong\mathfrak{g}^*//G,$ which gives us the desired map.

We now combine the conformal grading and the filtration $F_i$ to define the slope filtration on the space of opers. For $r$ a nonnegative rational number, define

\begin{equation*}
V_{\can}^*((t))_{>r}=\displaystyle\bigoplus_{i,j\in\mathbb{Z}_{> 0}\mid j>r\cdot i}F_iV_{\can}^*((t))^j
\end{equation*}

\noindent or, using (\ref{grading}),

\begin{equation}\label{explicit}
V_{\can}^*((t))_{>r}=\displaystyle\bigoplus_{n,k\in\mathbb{Z}_{\geq 0}\mid k>r(n+1)+n}V_{\can,n}^*((t))t^k.
\end{equation}

Let $\Op_{\check{\mathfrak{g}}}^{\leq r}$ be the vanishing locus of $V_{\can}^*((t))_{>r}$. It is the space of opers with slope at most $r$ in the sense of \cite{CK}. They make the remarkable observation that the slope of any local system can be computed explicitly from any oper structure on it. This is done via a simple formula, Definition 1 in loc.cit., which can be matched with the definition above.

\subsection{Statement of result}

The main result of this section ``morally'' states that $\hat{\mathfrak{g}}_{\crit}\operatorname{-mod}_{\leq r}$ can be identified with the subcategory of representations which, when viewed as a module over $Z(U(\hat{\mathfrak{g}})_{\crit})\cong\mathcal{O}(\Op_{\check{\mathfrak{g}}}),$ are (set-theoretically) supported on $\Op_{\check{\mathfrak{g}}}^{\leq r}.$ Because of renormalization issues, we need the machinery of \cite{semiinf} to define this latter category.

By Theorem 11.18.1 of \cite{semiinf}, there is a natural coaction of the coalgebra $\IndCoh^*(\Op_{\check{\mathfrak{g}}})$ on $\hat{\mathfrak{g}}_{\crit}\operatorname{-mod}$, and this coaction commutes with the $G((t))$-action. More explicitly, we have a coaction map

$$\hat{\mathfrak{g}}_{\crit}\operatorname{-mod}\rightarrow\hat{\mathfrak{g}}_{\crit}\operatorname{-mod}\otimes\IndCoh^*(\Op_{\check{\mathfrak{g}}})$$

\noindent along with the data of its compatibilities with the comultiplication map

$$\IndCoh^*(\Op_{\check{\mathfrak{g}}})\rightarrow\IndCoh^*(\Op_{\check{\mathfrak{g}}})\otimes\IndCoh^*(\Op_{\check{\mathfrak{g}}}).$$

The coaction encodes restriction of modules along the homomorphism

$$\mathcal{O}(\Op_{\check{\mathfrak{g}}})\otimes U(\hat{\mathfrak{g}})_{\crit}\rightarrow U(\hat{\mathfrak{g}})_{\crit}.$$

For any two comodules $M,N$ of $\IndCoh^*(\Op_{\check{\mathfrak{g}}})$, we can define a cotensor product $M\overset{\IndCoh^*(\Op_{\check{\mathfrak{g}}})}{\otimes}N$ as the totalization of the cosimplicial object beginning

$$M\otimes N\rightrightarrows M\otimes\IndCoh^*(\Op_{\check{\mathfrak{g}}})\otimes N\cdots.$$

Let $\widehat{\Op_{\check{\mathfrak{g}}}^{\leq r}}$ be the formal completion of $\Op_{\check{\mathfrak{g}}}$ along $\Op_{\check{\mathfrak{g}}}^{\leq r}$. Then we define

$$\hat{\mathfrak{g}}_{\crit}\operatorname{-mod}_{\widehat{\Op_{\check{\mathfrak{g}}}^{\leq r}}}\cong\hat{\mathfrak{g}}_{\crit}\operatorname{-mod}\overset{\IndCoh^*(\Op_{\check{\mathfrak{g}}})}{\otimes}\IndCoh^*(\widehat{\Op_{\check{\mathfrak{g}}}^{\leq r}}).$$

\begin{remark}
It is possible to dualize $\IndCoh^*(\Op_{\check{\mathfrak{g}}})$ to get an algebra (as opposed to coalgebra) $\IndCoh^!(\Op_{\check{\mathfrak{g}}})$, and to rewrite our definitions above in terms of modules rather than comodules. This is made explicit in Remark 6.4.1 of \cite{raskinFG}. Our choice, following loc.cit., to work with coalgebras has the advantage that the coaction map can be interpreted explicitly as a restriction map.
\end{remark}

\begin{theorem}\label{depthKM}
There is a natural equivalence of subcategories

$$\hat{\mathfrak{g}}_{\crit}\operatorname{-mod}_{\leq r}\cong\hat{\mathfrak{g}}_{\crit}\operatorname{-mod}_{\widehat{\Op_{\check{\mathfrak{g}}}^{\leq r}}}.$$
\end{theorem}

\subsection{Reduction step}

Let us start by reducing the theorem to an auxillary statement. To formulate it, note that, as the map $\IndCoh^*(\widehat{\Op_{\check{\mathfrak{g}}}^{=r}})\rightarrow\IndCoh^*(\Op_{\check{\mathfrak{g}}})$ is a composition of a fully faithful right adjoint and a fully faithful left adjoint, it remains fully faithful upon tensoring with any category.

\begin{theorem}\label{lemmadepthKM}
The coaction of $\IndCoh^*(\Op_{\check{\mathfrak{g}}})$ on $\hat{\mathfrak{g}}_{\crit}\operatorname{-mod}_{=r}$ factors through

$$\hat{\mathfrak{g}}_{\crit}\operatorname{-mod}_{=r}\rightarrow\hat{\mathfrak{g}}_{\crit}\operatorname{-mod}_{=r}\otimes\IndCoh^*(\widehat{\Op_{\check{\mathfrak{g}}}^{=r}}),$$

\noindent where $\widehat{\Op_{\check{\mathfrak{g}}}^{=r}}$ is the complement of $\widehat{\Op_{\check{\mathfrak{g}}}^{<r}}$ inside $\widehat{\Op_{\check{\mathfrak{g}}}^{\leq r}}$.
\end{theorem}

\begin{corollary}\label{cordepthKM}
The map

$$\hat{\mathfrak{g}}_{\crit}\operatorname{-mod}_{=s}\rightarrow\hat{\mathfrak{g}}_{\crit}\operatorname{-mod}_{=s}\overset{\IndCoh^*(\Op_{\check{\mathfrak{g}}})}{\otimes}\IndCoh^*(\widehat{\Op_{\check{\mathfrak{g}}}^{\leq r}})$$

\noindent given by cotensoring with the map $\IndCoh^*(\Op_{\check{\mathfrak{g}}})\rightarrow\IndCoh^*(\widehat{\Op_{\check{\mathfrak{g}}}^{\leq r}})$ is an equivalence if $s\leq r$, and zero if $s > r$.
\end{corollary}

\begin{proof}
Theorem \ref{lemmadepthKM} implies that the $\IndCoh^*(\Op_{\check{\mathfrak{g}}})$-comodule structure on $\hat{\mathfrak{g}}_{\crit}\operatorname{-mod}_{=s}$ descends to a structure of $\IndCoh^*(\widehat{\Op_{\check{\mathfrak{g}}}^{=s}})$-comodule. We can thus rewrite the map of interest as

$$\hat{\mathfrak{g}}_{\crit}\operatorname{-mod}_{=s}\rightarrow\hat{\mathfrak{g}}_{\crit}\operatorname{-mod}_{=s}\overset{\IndCoh^*(\widehat{\Op_{\check{\mathfrak{g}}}^{=s}})}{\otimes}\IndCoh^*(\widehat{\Op_{\check{\mathfrak{g}}}^{=s}})\overset{\IndCoh^*(\Op_{\check{\mathfrak{g}}})}{\otimes}\IndCoh^*(\widehat{\Op_{\check{\mathfrak{g}}}^{\leq r}}).$$

The corollary now follows from the fact that

$$\IndCoh^*(\widehat{\Op_{\check{\mathfrak{g}}}^{=s}})\overset{\IndCoh^*(\Op_{\check{\mathfrak{g}}})}{\otimes}\IndCoh^*(\widehat{\Op_{\check{\mathfrak{g}}}^{\leq r}})$$

\noindent is trivial for $s>r$, and isomorphic to $\IndCoh^*(\widehat{\Op_{\check{\mathfrak{g}}}^{=s}})$ for $s\leq r.$
\end{proof}

To prove Theorem \ref{depthKM} assuming Theorem \ref{lemmadepthKM}, we show that we have a natural equivalence of subcategories, for all $s\in\mathbb{R}$,

$$\hat{\mathfrak{g}}_{\crit}\operatorname{-mod}_{\leq\operatorname{min}(r,s)}\cong\hat{\mathfrak{g}}_{\crit}\operatorname{-mod}_{\leq s}\overset{\IndCoh^*(\Op_{\check{\mathfrak{g}}})}{\otimes}\IndCoh^*(\widehat{\Op_{\check{\mathfrak{g}}}^{\leq r}}).$$

Taking the colimit as $s\to\infty$, we get the desired statement. (Here we rewrite the colimit as a limit along right adjoints to commute it with the cotensor product on the RHS.) We prove this statement by induction on $s$, recalling that the depth filtration is discretely indexed. The base case, of $s$ negative, is obvious. Assume it holds for all $s' < s$.

We split into two cases, of $s \leq r$ and $s > r$. Let us treat the $s\leq r$ case; the other one is similar and is left to the reader. We invoke the following commutative diagram:

\begin{equation*}
\begin{tikzcd}
\hat{\mathfrak{g}}_{\crit}\operatorname{-mod}_{< s} \arrow[d] \arrow[r] & \hat{\mathfrak{g}}_{\crit}\operatorname{-mod}_{< s}\overset{\IndCoh^*(\Op_{\check{\mathfrak{g}}})}{\otimes}\IndCoh^*(\widehat{\Op_{\check{\mathfrak{g}}}^{\leq r}}) \arrow[d] \\
\hat{\mathfrak{g}}_{\crit}\operatorname{-mod}_{\leq s} \arrow[r] & \hat{\mathfrak{g}}_{\crit}\operatorname{-mod}_{\leq s}\overset{\IndCoh^*(\Op_{\check{\mathfrak{g}}})}{\otimes}\IndCoh^*(\widehat{\Op_{\check{\mathfrak{g}}}^{\leq r}})\\
\hat{\mathfrak{g}}_{\crit}\operatorname{-mod}_{= s} \arrow[u] \arrow[r] & \hat{\mathfrak{g}}_{\crit}\operatorname{-mod}_{= s}\overset{\IndCoh^*(\Op_{\check{\mathfrak{g}}})}{\otimes}\IndCoh^*(\widehat{\Op_{\check{\mathfrak{g}}}^{\leq r}}) \arrow[u].
\end{tikzcd}
\end{equation*}

The top arrow is an equivalence, by the induction hypothesis, and the bottom arrow is an equivalence by Corollary \ref{cordepthKM}. The vertical arrows are all fully faithful, and by Lemma \ref{inductionaux}, the middle row is generated by the essential images of the top and bottom row, so the middle arrow is also an equivalence, as desired.

Our only remaining task is to prove Theorem \ref{lemmadepthKM}, which will take up most of the rest of this section.

\subsection{Proof of Theorem \ref{lemmadepthKM}}

\subsubsection{Generators for $\hat{\mathfrak{g}}_{\crit}\operatorname{-mod}^{K_{x,r+},\circ}$}
\label{KMstepgenerators}

\begin{proof}
Start by assuming $r>0$; the modifications necessary for the $r=0$ case will be explained at the end of the proof.

By Theorem \ref{generation} it suffices to show that, for any $x\in X_*(T)\otimes\mathbb{R},$ the coaction of $\IndCoh^*(\Op_{\check{\mathfrak{g}}})$ maps $\hat{\mathfrak{g}}_{\crit}\operatorname{-mod}^{K_{x,r+},\circ}$ into $\hat{\mathfrak{g}}_{\crit}\operatorname{-mod}\otimes\IndCoh^*(\widehat{\Op_{\check{\mathfrak{g}}}^{=r}}).$ Furthermore, it suffices to check this on a set of generators for $\hat{\mathfrak{g}}_{\crit}\operatorname{-mod}^{K_{x,r+},\circ}.$

We claim that for any cover of $(K_{x,r}/K_{x,r+})^{*,\circ}$ by affine opens $U_i$, $\hat{\mathfrak{g}}_{\crit}\operatorname{-mod}^{K_{x,r+},\circ}$ is generated by the modules $\operatorname{Ind}_{\mathfrak{k}_{x,r}}^{\hat{\mathfrak{g}}}\mathcal{F}$, with $\mathcal{F}$ ranging over all objects in the $\QCoh(U_i)^{\heartsuit}$. (Here the $\mathfrak{k}_{x,r}$ action is through $\mathfrak{k}_{x,r}/\mathfrak{k}_{x,r+}\cong K_{x,r}/K_{x,r+}$.) We will later choose specific $U_i$ --- until then, all our statements will hold for any affine open cover.

We show this by iteratively producing generators for various categories.\footnote{The functors we use are constructed in Section 5 of \cite{gurbirjustin}. This section only appears in the original arxiv version; it is no longer needed in the current version (which also adds Sam Raskin to the list of authors). We thank Gurbir Dhillon for pointing us to this reference.} First we note that $\QCoh(U_i)$ is generated by $\QCoh(U_i)^{\heartsuit}$ and $\QCoh((K_{x,r}/K_{x,r+})^{*,\circ})$ is generated by the $\QCoh(U_i)$. Thus, $\QCoh((K_{x,r}/K_{x,r+})^{*,\circ})$ is generated by the $\mathcal{F}$.

Next we produce generators for the category $\mathfrak{k}_{x,r}\operatorname{-mod}^{K_{x,r+},\circ}$, defined via 
$$\mathfrak{k}_{x,r}\operatorname{-mod}^{K_{x,r+},\circ}\cong\mathfrak{k}_{x,r}\operatorname{-mod}^{K_{x,r+}}\otimes_{D((K_{x,r}/K_{x,r+})^*)}D((K_{x,r}/K_{x,r+})^{*,\circ}).$$

We have an inflation functor

$$\QCoh((K_{x,r}/K_{x,r+})^{*})\cong\mathfrak{k}_{x,r}/\mathfrak{k}_{x,r+}\operatorname{-mod}\rightarrow\mathfrak{k}_{x,r}\operatorname{-mod}^{K_{x,r+}},$$

\noindent given by restriction along the morphism $\mathfrak{k}_{x,r}\rightarrow\mathfrak{k}_{x,r}/\mathfrak{k}_{x,r+}$, and the image of this functor generates $\mathfrak{k}_{x,r}\operatorname{-mod}^{K_{x,r+}}$. As having generating image for a functor is equivalent to having trivial quotient, this property is preserved under tensor products. Using Lemma \ref{appendixtensor}, which gives $D((K_{x,r}/K_{x,r+})^{*,\circ})\otimes_{D((K_{x,r}/K_{x,r+})^*)}\QCoh((K_{x,r}/K_{x,r+})^*)\cong\QCoh((K_{x,r}/K_{x,r+})^{*,\circ})$, the image of the functor

$$\QCoh((K_{x,r}/K_{x,r+})^{*,\circ})\rightarrow\mathfrak{k}_{x,r}\operatorname{-mod}^{K_{x,r+},\circ}$$

\noindent also generates.

Finally, we have a $K_{x,r}$-equivariant pair of adjoint functors

\begin{equation*}
\begin{tikzcd}
\mathfrak{k}_{x,r}\operatorname{-mod}\arrow[r, shift left=.75ex, "\operatorname{Ind}"] & \hat{\mathfrak{g}}_{\crit}\operatorname{-mod}\arrow[l, shift left=.75ex, "\operatorname{Res}"]  
\end{tikzcd}
\end{equation*}

\noindent which induces an adjoint pair

\begin{equation*}
\begin{tikzcd}
\mathfrak{k}_{x,r}\operatorname{-mod}^{K_{x,r+},\circ}\arrow[r, shift left=.75ex, "\operatorname{Ind}'"] & \hat{\mathfrak{g}}_{\crit}\operatorname{-mod}^{K_{x,r+},\circ}\arrow[l, shift left=.75ex, "\operatorname{Res}'"].
\end{tikzcd}
\end{equation*}

As before, to show that $\operatorname{Ind}'$ has generating image, it suffices to show the same for $\operatorname{Ind}$. This in turn is equivalent to conservativity of $\operatorname{Res}$, which is proven in loc.cit. Combining our generation statements, we see that the $\operatorname{Ind}_{\mathfrak{k}_{x,r}}^{\hat{\mathfrak{g}}}\mathcal{F}$ generate, as desired.

Therefore, it suffices to know that the $\operatorname{coact}(\operatorname{Ind}_{\mathfrak{k}_{x,r}}^{\hat{\mathfrak{g}}}\mathcal{F})$ lies inside \\ $\hat{\mathfrak{g}}_{\crit}\operatorname{-mod}\otimes\IndCoh^*(\widehat{\Op_{\check{\mathfrak{g}}}^{=r}}).$ By Lemma 11.13.1 of \cite{semiinf}, we know that that $\operatorname{coact}(\operatorname{Ind}_{\mathfrak{k}_{x,r}}^{\hat{\mathfrak{g}}}\mathcal{F})\in(\hat{\mathfrak{g}}_{\crit}\operatorname{-mod}\otimes\IndCoh^*(\Op_{\check{\mathfrak{g}}}))^{\heartsuit},$ so we can use abelian-categorical methods.

\subsubsection{Upper bound on depth}
\label{KMstepupper}

We start by showing that $\operatorname{coact}(\operatorname{Ind}_{\mathfrak{k}_{x,r}}^{\hat{\mathfrak{g}}}\mathcal{F})\in(\hat{\mathfrak{g}}_{\crit}\operatorname{-mod}\otimes\IndCoh^*(\Op_{\check{\mathfrak{g}}}^{\leq r}))^{\heartsuit}.$ As we are in abelian-category land, it suffices to show this after forgetting the $\hat{\mathfrak{g}}_{\crit}$-module structure. Concretely, we need to prove that the action of $V^*_{\can}((t))_{>r}\subseteq Z(U(\hat{\mathfrak{g}})_{\crit})$ on $\operatorname{Ind}_{\mathfrak{k}_{x,r}}^{\hat{\mathfrak{g}}}\mathcal{F}$ is trivial.

Pick positive integers $i$ and $j$ with $j>r\cdot i.$ Interpreting $F_iV^*_{\can}((t))^j$ as a subspace of $Z(U(\hat{\mathfrak{g}})_{\crit}),$ we can write any element $a\in F_iV^*_{\can}((t))^j$ as a sum of monomials

$$a=\sum_d c_d\cdot a_{d,1}t^{b_{d,1}}\cdot a_{d,2}t^{b_{d,2}}\cdots a_{d,n_d}t^{b_{d,n_d}},$$

\noindent where each $c_d$ is an element of $k$ and each $a_{d,e}$ is an element of some $\mathfrak{g}_{\alpha_{d,e}}$. For any $d$, the compatibilities described earlier tell us that $n_d\leq i$ and $b_{d,1}+b_{d,2}+\cdots+b_{d,n_d}=j,$ and the centrality of $a$ implies that $\displaystyle\sum_e\alpha_{d,e}=0$. Furthermore, by applying commutation relations, we can choose the monomials so that

$$\langle\alpha_{d,1},x\rangle+b_{d,1}\leq\langle\alpha_{d,2},x\rangle+b_{d,2}\leq\cdots\leq\langle\alpha_{d,n_d},x\rangle+b_{d,n_d}.$$

As we have

$$\sum_e\langle\alpha_{d,e},x\rangle+b_{d,e}=\langle 0,x\rangle+j=j,$$

\noindent this implies that

$$\langle\alpha_{d,n_d},x\rangle+b_{d,n_d}\geq\frac{j}{i} > r,$$

\noindent or equivalently,

$$a_{d,n_d}t^{b_{d,n_d}}\in\mathfrak{k}_{x,r+}.$$

Therefore, $V^*_{\can}((t))_{>r}\subseteq U(\hat{\mathfrak{g}})_{\crit}\mathfrak{k}_{x,r+}.$ But this immediately implies that any element $a\in V^*_{\can}((t))_{>r}$ acts trivially on $\mathcal{F}\subseteq\operatorname{Ind}_{\mathfrak{k}_{x,r}}^{\hat{\mathfrak{g}}}\mathcal{F}.$ Because $a$ is central and $U(\hat{\mathfrak{g}})\otimes\mathcal{F}\rightarrow\operatorname{Ind}_{\mathfrak{k}_{x,r}}^{\hat{\mathfrak{g}}}\mathcal{F}$ is surjective, $a$ acts trivially on the entirety of $\operatorname{Ind}_{\mathfrak{k}_{x,r}}^{\hat{\mathfrak{g}}}\mathcal{F}$.

\subsubsection{Lower bound on depth}
\label{lowerbound}

Let $W$ be the affine space $\Op_{\check{\mathfrak{g}}}^{\leq r}/\Op_{\check{\mathfrak{g}}}^{< r}$. Choose a covering of the complement of the origin $W\setminus 0$ by open affines $U_i'$. Now we claim we can choose the original open cover $U_i$ so that, for $\mathcal{F}\in\QCoh(U_i)^{\heartsuit}$, we have

$$\operatorname{coact}(\operatorname{Ind}_{\mathfrak{k}_{x,r}}^{\hat{\mathfrak{g}}}\mathcal{F})\in(\hat{\mathfrak{g}}_{\crit}\operatorname{-mod}\otimes\IndCoh^*(\Op_{\check{\mathfrak{g}}}^{\leq r}\times_{W}U'_i))^{\heartsuit},$$

\noindent which will imply Theorem \ref{lemmadepthKM} and hence Theorem \ref{depthKM}. To define $U_i$ and prove this containment, we repeat the previous analysis for $V^*_{\can}((t))_{\geq r}\subseteq Z(U(\hat{\mathfrak{g}})_{\crit})$ in place of $V^*_{\can}((t))_{>r}.$ We keep the notation from the previous section. For a monomial $a_{d,1}t^{b_{d,1}}\cdot a_{d,2}t^{b_{d,2}}\cdots a_{d,n_d}t^{b_{d,n_d}}$ to not act trivially on $\operatorname{Ind}_{\mathfrak{k}_{x,r}}^{\hat{\mathfrak{g}}}\mathcal{F}$, each of the inequalities in the previous section must be an equality. Namely, we must have, for all $d$ and $e$, $n_d=i$ and

$$\langle\alpha_{d,e},x\rangle+b_{d,e}=\frac{j}{i}=r,$$

\noindent which implies that we have $a_{d,e}t^{b_{d,e}}\in\mathfrak{k}_{x,r}.$ In other words, we have an inclusion $V^*_{\can}((t))_{\geq r}\subseteq U(\hat{\mathfrak{g}})_{\crit}\mathfrak{k}_{x,r+}+U(\mathfrak{k}_{x,r})$, which induces a map of algebras

$$\operatorname{Sym}(V^*_{\can}((t))_{\geq r}/V^*_{\can}((t))_{>r})\rightarrow\operatorname{Sym}(\mathfrak{k}_{x,r}/\mathfrak{k}_{x,r+})$$

Noting that

$$W\cong\operatorname{Spec}\operatorname{Sym}(V^*_{\can}((t))_{\geq r}/V^*_{\can}((t))_{>r}),$$

\noindent this becomes a map of varieties $f:(K_{x,r}/K_{x,r+})^*\rightarrow W,$ and the $\mathcal{O}(W)$-module structure on $\mathcal{F}\subseteq\operatorname{Ind}_{\mathfrak{k}_{x,r}}^{\hat{\mathfrak{g}}}\mathcal{F}$ corresponds to the pushforward $f_*\mathcal{F}$. This shows our claim for $U_i=f^{-1}(U'_i)$, once we show that the $U_i$ are an open cover of $(K_{x,r}/K_{x,r+})^{*,\circ}$. Equivalently, we need to show the equality of closed subvarieties $f^{-1}(0)=(K_{x,r}/K_{x,r+})^*_{us}.$

To analyze $f$, we reinterpret our objects in terms of graded Lie algebras. Write $r$ as a fraction $\frac{p}{q}$ with $p,q$ relatively prime positive integers. As in (\ref{explicit}), we have

\begin{align*}
V_{\can}^*((t))_{>r}=&\displaystyle\bigoplus_{n,k\in\mathbb{Z}_{\geq 0}\mid k>r(n+1)+n}V_{\can,n}^*((t))t^k\\
V_{\can}^*((t))_{\geq r}=&\displaystyle\bigoplus_{n,k\in\mathbb{Z}_{\geq 0}\mid k\geq r(n+1)+n}V_{\can,n}^*((t))t^k\\
V_{\can}^*((t))_{\geq r}/V_{\can}^*((t))_{>r}\cong&\displaystyle\bigoplus_{n,k\in\mathbb{Z}_{\geq 0}\mid k=r(n+1)+n}V_{\can,n}^*((t))t^k\\
\cong&\displaystyle\bigoplus_{n\in\mathbb{Z}_{\geq 0}\mid n\equiv -1\pmod{q}} V^*_{\can,n}.
\end{align*}

Dualizing, we have an identification $W\cong V_{\can,q-1}\oplus V_{\can,2q-1}\oplus\cdots$. For $i=0,1,\cdots, q-1$, define

$$\mathfrak{g}_i=\displaystyle\bigoplus_{\langle x,\alpha\rangle\equiv\frac{mi}{n}\pmod{1}}\mathfrak{g}_{\alpha}.$$

This gives a grading on a Lie algebra $\mathfrak{g}_0\oplus\cdots\oplus\mathfrak{g}_{q-1}$ (possibly non-isomorphic to $\mathfrak{g}$.) Note that we can identify $K_{x,r}/K_{x,r+}$ with $\mathfrak{g}_1$ and $L_x$ with the exponentiation $G_0$ of $\mathfrak{g}_0$, and under this identification, the $L_x$ action on $K_{x,r}/K_{x,r+}$ becomes the adjoint action of $G_0$ on $\mathfrak{g}_1$. In this language, compatibility of Feigin-Frenkel with the symbol map (see the discussion surrounding (\ref{compatibility})) becomes commutativity of the following diagram, where the bottom map (denoted $i$) is induced from the inclusion $\mathfrak{g}^{*}_1\rightarrow\mathfrak{g}^*$:

\begin{equation*}
\begin{tikzcd}
\mathfrak{g}^{*}_1 \arrow[r, "f"] \arrow[d, "g"] & W \arrow[d] \\
\mathfrak{g}^{*}_1//G_0 \arrow[r, "i"] & \mathfrak{g}^*//G.
\end{tikzcd}
\end{equation*}

As $W\rightarrow\mathfrak{g}^*//G\cong V_{\can}$ is a closed embedding, we see that $f$ factors as the composition

\begin{equation*}
\begin{tikzcd}
\mathfrak{g}^{*}_1\arrow[r, "g"] & \mathfrak{g}^{*}_1//G_0 \arrow[r, "h"] & W.
\end{tikzcd}
\end{equation*}

By generalities on GIT quotients, the preimage $g^{-1}(0)$ can be identified with $(\mathfrak{g}_1)^*_{us}.$ Therefore, it suffices to show that $h^{-1}(0)=0$. We show:

\begin{lemma}
The map $h$ is finite. In particular, as the only finite $\mathbb{G}_m$-invariant subvariety of $\mathfrak{g}^{*}_1//G_0$ is the origin, $h^{-1}(0)=0$.
\end{lemma}

\begin{proof}
As $i$ is just $h$ composed with a closed embedding, we may show that $i$ is finite instead. To do so, we use some results from \cite{Vinberg}. We recall that a maximal subspace $\mathfrak{c}\subseteq\mathfrak{g}_1$ consisting of commuting semisimple elements is called a Cartan subspace. Choose such a subspace, as well as an extension $\mathfrak{h}\subseteq\mathfrak{g}$ to a Cartan of $\mathfrak{g}$. Then we have a commutative diagram

\begin{equation*}
\begin{tikzcd}
\mathfrak{c} \arrow[r] \arrow[d] & \mathfrak{h}\arrow[d]\\
\mathfrak{g}_1^*//G_0 \arrow[r, "i"] & \mathfrak{g}^*//G.
\end{tikzcd}
\end{equation*}

By Theorem 7 of \cite{Vinberg}, the vertical arrows are finite and surjective. (In fact, they are global quotients by finite groups.) The top arrow is evidently a closed embedding, hence finite. Commutativity of the diagram now implies that $i$ is finite, as desired.

\end{proof}

\subsubsection{Case of $r=0$}

Let us explain how to modify the proof to work when $r=0$. This case is substantially simpler, because $\hat{\mathfrak{g}}_{\crit}\operatorname{-mod}^{K_{x,0+},\circ}\cong\hat{\mathfrak{g}}_{\crit}\operatorname{-mod}^{K_{x,0+}}.$ In particular, we do not need the argument of Section \ref{lowerbound}.

We start by constructing generators for $\hat{\mathfrak{g}}_{\crit}\operatorname{-mod}^{K_{x,0+}}.$ The arguments of Section \ref{KMstepgenerators} show that the composition

\[\mathfrak{k}_{x,0}/\mathfrak{k}_{x,0+}\operatorname{-mod}\rightarrow\mathfrak{k}_{x,0}\operatorname{-mod}^{K_{x,0+}}\rightarrow\hat{\mathfrak{g}}_{\crit}\operatorname{-mod}^{K_{x,0+}}\]

\noindent has generating image. As $\mathfrak{k}_{x,0}/\mathfrak{k}_{x,0+}\operatorname{-mod}$ is the derived category of modules over an algebra, it is generated by the elements of its heart.

It thus suffices to show that for any $M\in\mathfrak{k}_{x,0}/\mathfrak{k}_{x,0+}\operatorname{-mod}^{\heartsuit},$ we have

\[\operatorname{coact}(\operatorname{Ind}_{\mathfrak{k}_{x,0}}^{\hat{\mathfrak{g}}}M)\in\hat{\mathfrak{g}}_{\crit}\operatorname{-mod}\otimes\IndCoh^*(\widehat{\Op_{\check{\mathfrak{g}}}^{=0}}).\]

This follows from the same argument as in Section \ref{KMstepupper}.
\end{proof}

\subsection{Remarks and extensions}

\subsubsection{Relation to Chen-Kamgarpour}

Theorem \ref{depthKM} is closely related to the work of Chen-Kamgarpour in \cite{CK}. The first half of the proof of Theorem \ref{lemmadepthKM} (in which we show the upper bound on depth) essentially repeats the proof of Theorem 1 in loc.cit. However, the proof of the lower bound (which critically uses Theorem \ref{generation}) is novel and allows us to significantly strengthen their results.

In particular, Theorem \ref{depthKM} implies Conjecture 1 of \cite{CK}. Let us state their conjecture in our language. Let $\chi$ be some $\check{\mathfrak{g}}$-oper on the punctured disk. Then we can define a category

$$\hat{\mathfrak{g}}_{\crit}\operatorname{-mod}_{\chi}\cong\hat{\mathfrak{g}}_{\crit}\operatorname{-mod}\overset{\IndCoh^*(\Op_{\check{\mathfrak{g}}})}{\otimes}\IndCoh^*(\operatorname{Spec} k_{\chi}).$$

Let $r$ be the depth of $\xi$ as an oper. Their Conjecture 1 states that
\begin{itemize}
\item $(\hat{\mathfrak{g}}_{\crit}\operatorname{-mod}_{\chi})_{<r}$ is trivial
\item $(\hat{\mathfrak{g}}_{\crit}\operatorname{-mod}_{\chi})_{\leq r}$ is nontrivial
\end{itemize}

\noindent and they prove the first statement as their Theorem 1. Our Theorem \ref{depthKM} implies that

$$(\hat{\mathfrak{g}}_{\crit}\operatorname{-mod}_{\chi})_{\leq r}\cong\hat{\mathfrak{g}}_{\crit}\operatorname{-mod}_{\chi}.$$

This is actually much stronger than their second statement --- it contains, for instance, a localization statement (see the next subsubsection.)

\subsubsection{Relation to Frenkel-Gaitsgory}

Our theorem can be thought of as providing a ramified analogue of a conjecture of Frenkel-Gaitsgory from \cite{FG2}. They define a category of representations $\hat{\mathfrak{g}}_{\crit}\operatorname{-mod}_{\operatorname{reg}}$ with central support on the locus of regular opers $\Op_{\check{\mathfrak{g}}}^{\operatorname{reg}},$ and they conjecture an equivalence with a (modified) category of D-modules. The case of $\operatorname{SL}_2$ was proven by Raskin in \cite{raskinFG}.

Similarly, for $r$ a nonnegative integer $n$, it can (nontrivially) be deduced from Theorem \ref{depthKM} that

$$\hat{\mathfrak{g}}_{\crit}\operatorname{-mod}_{\widehat{\Op_{\check{\mathfrak{g}}}^{\leq n}}}\cong D(G((t))/K_{0,n+1})\otimes_{D(K_{0,n+1}\backslash G((t))/K_{0,n+1})}\hat{\mathfrak{g}}_{\crit}\operatorname{-mod}^{K_{0,n+1}}.$$

We would like to ``set $n=-1$'' to prove a localization statement for the locus of regular opers. With some work this can be done, but there are serious technical difficulties involved in going from the resulting statement to the Frenkel-Gaitsgory conjecture. For $\operatorname{SL}_2$, a workaround for these obstacles was developed in \cite{raskinFG}. In a future paper, joint with S. Raskin, we plan to combine the methods of \cite{raskinFG} with the Moy-Prasad theory to prove the Frenkel-Gaitsgory conjecture for all $G$.

\subsubsection{Denominators of depths}

Finally, let us isolate another consequence of our argument. Choose a set of homogeneous generators of $U(\mathfrak{g})^G$ as a polynomial algebra. Their degrees will be called the fundamental degrees of $\mathfrak{g}$.

\begin{proposition}\label{divisibility}
Let $C$ be a $G((t))$-category and $r$ a rational number such that $C_{=r}$ is nontrivial. Then the denominator of $r$ must divide a fundamental degree of $\mathfrak{g}$. 
\end{proposition}

This is essentially not new (see Remark 2 of \cite{CK}, where it is mentioned that this was proven by J.K.Yu in the arithmetic setting.)

\begin{proof}
By Theorem \ref{generation}, we must have that $(K_{x,r}/K_{x,r+})^{*,\circ}$ is nontrivial for some $x$. This implies, in the language of graded Lie algebras, that $\mathfrak{g}_1^*/G_0$ is nontrivial, and arguing as in the end of our proof shows that the denominator $q$ of $r$ must divide the degree of some generator of $\mathcal{O}(\mathfrak{g}^*//G).$
\end{proof}

\section{Depth filtration on $\Whit$}

\subsection{The category $\Whit$}

We now compute the depth filtration on the category of Whittaker sheaves. This category, denoted $\Whit$, is defined as the category of D-modules on $G((t))$ which are right $N((t))$-equivariant with respect to a certain character $\psi: N((t))\rightarrow k.$ To define this character, note that $N/[N,N]$ is a sum of copies of the additive group, and so we have maps

\[N((t))\rightarrow N/[N,N]((t))\rightarrow N/[N,N],\]

\noindent where the last map is the residue map. Composing with any nondegenerate (i.e., nontrivial on every simple root) character of $N/[N,N]$, we get the desired character $\psi$.

Our main result is that the truncations $\Whit_{\leq r}$ can also be expressed as categories of equivariant D-modules on $G((t))$. For $r$ integral, we recover the adolescent Whittaker construction of \cite{Whit}, as anticipated by Remark 1.22.2 of loc.cit. This section can thus be viewed as an extension of the adolescent formalism to rational $r$, and our results can likely be proven via the methods of \cite{Whit}. However, we will take a different approach, using the (categorical) representation theory of the Heisenberg group.

\subsection{Definition of subgroups}

For any rational number $s\geq 1$, we define a subgroup $\overset{\circ}{I}_s\subseteq G((t))$. It will be expressable as a product $B^{-}[[t]]_sN((t))_s$, for subgroups $B^{-}[[t]]_s\subseteq B^{-}[[t]]$ and $N((t))_s\subseteq N((t))$. We can describe these in terms of Moy-Prasad subgroups:

\[B^{-}[[t]]_s=K_{s\check{\rho},(s-1)+}\cap B^{-}[[t]]\]

\noindent and

\[N((t))_s=P_{s\check{\rho}}\cap N((t))\]

\noindent where $\check{\rho}$ denotes the half-sum of positive coroots. In fact, we prove a slightly more general statement than that $\overset{\circ}{I}_s$ is a subgroup.

\begin{lemma}
Assume we have rational numbers $s'\geq s\geq 1.$ Then $B^{-}[[t]]_{s'}N((t))_s$ is a subgroup of $G((t))$.
\end{lemma}

\begin{proof}
It suffices to show that $[B^{-}[[t]]_{s'},N((t))_s]\subseteq B^{-}[[t]]_{s'}N((t))_s.$ Note that

\begin{align*}
[B^{-}[[t]]_{s'},N((t))_s]&\subseteq [B^{-}[[t]]_s,N((t))_s]\cap [B^{-}[[t]]_{s'},N((t))_{s'}]\\
&\subseteq [K_{s\check{\rho},(s-1)+},P_{s\check{\rho}}]\cap [K_{s'\check{\rho},(s'-1)+},P_{s'\check{\rho}}]\\
&\subseteq K_{s\check{\rho},(s-1)+}\cap K_{s'\check{\rho},(s'-1)+}.
\end{align*}

As for the $K_{x,r+}$, we can factor $K_{s\check{\rho},(s-1)+}\cap K_{s'\check{\rho},(s'-1)+}$ into its intersections with $B^{-}[[t]]$ and $N((t))$. We have

\begin{align*}
K_{s\check{\rho},(s-1)+}\cap K_{s'\check{\rho},(s'-1)+}\cap B^{-}[[t]]&\cong B^{-}[[t]]_{s}\cap B^{-}[[t]]_{s'}\\
&\cong B^{-}[[t]]_{s'}
\end{align*}

\noindent and

\begin{align*}
K_{s\check{\rho},(s-1)+}\cap K_{s'\check{\rho},(s'-1)+}\cap N((t))&\subseteq P_{s\check{\rho}}\cap P_{s'\check{\rho}}\cap N((t))\\
&\cong N((t))_s\cap N((t))_{s'}\\
&\cong N((t))_s
\end{align*}

\noindent so $K_{s\check{\rho},(s-1)+}\cap K_{s'\check{\rho},(s'-1)+}\subseteq B^{-}[[t]]_{s'}N((t))_s$, as desired.
\end{proof}

We would like to define from $\psi$ a character on $\overset{\circ}{I}_s$. This is accomplished by the following lemma.

\begin{lemma}
The restriction $\psi |_{N((t))_s}$extends uniquely to a character on $\overset{\circ}{I}_s$ trivial on $B^{-}[[t]]_s$.
\end{lemma}

\begin{proof}
We need to check that $\psi$ kills $[\overset{\circ}{I}_s,\overset{\circ}{I}_s]\cap N((t))$. As $[B^{-}[[t]]_s,B^{-}[[t]]_s]\cap N((t))]$ is trivial and $[N((t))_s,N((t))_s]\cap N((t))\subseteq [N,N]((t))\subseteq \operatorname{ker}{\psi},$ it suffices to show that

\[[B^{-}[[t]]_s, N((t))_s]\cap N((t))\subseteq\operatorname{ker}{\psi}.\]

To do so, note that

\begin{align*}
[B^{-}[[t]]_s,N((t))_s]\cap N((t))&\subseteq[K_{s\check{\rho},(s-1)+},P_{s\check{\rho}}]\cap N((t))\\
&\subseteq K_{s\check{\rho},(s-1)+}\cap N((t)),
\end{align*}

\noindent so it suffices to show that $K_{s\check{\rho},(s-1)+}\cap N((t))\subseteq N[[t]]+[N,N]((t))$ (which lies in the kernel of $\psi$ by definition.) We can check this with a short computation at the level of Lie algebras:

\begin{align*}
\mathfrak{k}_{s\check{\rho},(s-1)+}\cap\mathfrak{n}((t))&=\displaystyle\bigoplus_{\substack{\langle\alpha,s\check{\rho}\rangle+i>s-1\\ \langle\alpha,\check{\rho}\rangle>0}}\mathfrak{g}_{\alpha}t^i\\
&\subseteq\displaystyle\bigoplus_{\substack{\langle s\check{\rho},x\rangle+i>s-1\\ \langle\alpha,\check{\rho}\rangle=1}}\mathfrak{g}_{\alpha}t^i+[\mathfrak{n},\mathfrak{n}]((t))\\
&=\mathfrak{n}[[t]]+[\mathfrak{n},\mathfrak{n}]((t))
\end{align*}
\end{proof}

We will abuse notation and denote all of these characters by $\psi.$ For $r$ integral, our subgroups $\overset{\circ}{I}_s$ recover those of \cite{Whit}, and our characters $\psi$ coincide with those of loc.cit. So for instance, for $r=1$, taking $(\overset{\circ}{I}_1,\psi)$ invariants gives the baby Whittaker construction of \cite{ArkB} (at least when the center of $G$ is connected.)

Our main theorem is the following:

\begin{theorem}
We have an equivalence

$$\Whit_{\leq r}\cong D(G((t)))^{\overset{\circ}{I}_{r+1},\psi}.$$
\end{theorem}

\subsection{Reduction to inductive step}

To prove this, we analyze the maps between the $D(G((t)))^{\overset{\circ}{I}_{s},\psi}$. Call $s$ a jump locus if, for all rational numbers $s' < s$, $\overset{\circ}{I}_{s'}\neq\overset{\circ}{I}_{s}.$ Choose two consecutive jump loci $s_1<s_2$. Then we have a map

$$\iota_{s_1,s_2,*}:D(G((t)))^{\overset{\circ}{I}_{s_1},\psi}\rightarrow D(G((t)))^{\overset{\circ}{I}_{s_2},\psi}$$

\noindent given by

$$D(G((t)))^{B^{-}[[t]]_{s_1}N((t))_{s_1},\psi}\xrightarrow{\operatorname{Oblv}}D(G((t)))^{B^{-}[[t]]_{s_2}N((t))_{s_1},\psi}\xrightarrow{\operatorname{Av}_*}D(G((t)))^{B^{-}[[t]]_{s_2}N((t))_{s_2},\psi}.$$

By Lemma 2.9.1 and Theorem 2.1.1 of \cite{Whit}, the colimit of the $D(G((t)))^{\overset{\circ}{I}_{s},\psi}$ under the $\iota_{s_1,s_2,*}$ is equivalent to $\Whit$. Thus, it suffices to prove that

$$D(G((t)))^{\overset{\circ}{I}_{s},\psi}_{\leq r}\cong D(G((t)))^{\overset{\circ}{I}_{\operatorname{min}(r+1, s)},\psi}.$$

We do so by induction on $s$. The base case of $s=1$ is easy, as $\overset{\circ}{I}_1$ contains $K_{\check{\rho},0+}$ and hence $D(G((t)))^{\overset{\circ}{I}_{1},\psi}$ is purely of depth $0$. For the inductive step, again choose two consecutive jump loci $s_1<s_2$. We need to prove:

\begin{theorem}\label{consecutivejumploci}
The map $\iota_{s_1,s_2,*}$ is fully faithful with quotient purely of depth $s_2-1$.
\end{theorem}

\subsection{Proof of Theorem \ref{consecutivejumploci}}
\subsubsection{Preliminary definitions}

\begin{proof}
By the definition of a jump locus, we must have $B^{-}[[t]]_{s_1}=B^{-}[[t]]_{s_2-\epsilon}$ for sufficiently small $\epsilon > 0$. Thus, we have an equality

\begin{align*}
B^{-}[[t]]_{s_1}&=B^{-}[[t]]_{s_2-\epsilon}\\
&=K_{(s_2-\epsilon)\check{\rho},(s_2-\epsilon-1)+}\cap B^{-}[[t]]\\
&=K_{s_2\check{\rho},s_2-1}\cap B^{-}[[t]].
\end{align*}

In particular, $B^{-}[[t]]_{s_2}$ is a normal subgroup of $B^{-}[[t]]_{s_1},$ and the quotient (which we denote $U$) $B^{-}[[t]]_{s_1}/B^{-}[[t]]_{s_2}$ is a vector space. Namely, we see that $U$ is the image of $B^{-}[[t]]_{s_1}$ inside $K_{s_2\check{\rho},s_2-1}/K_{s_2\check{\rho},(s_2-1)+}.$ This can be rephrased as the following pullback square:

\begin{equation*}
\begin{tikzcd}
U \arrow[r] \arrow[d] & K_{s_2\check\rho,s_2-1}/K_{s_2\check\rho,(s_2-1)+} \arrow[d]\\
\mathfrak{b}^{-} \arrow[r] & \mathfrak{g}.
\end{tikzcd}
\end{equation*}

Similarly, we have

$$N((t))_{s_1}=K_{s_2\check{\rho},0+}\cap N((t))$$

\noindent and $N((t))_{s_1}$ is again a normal subgroup of $N((t))_{s_2}$. Denoting the quotient by $V$, we see that $V$ is the image of $N[[t]]_{s_2}$ inside $P_{s_2\check{\rho}}/K_{s_2\check{\rho},0+}$. We will often identify $U$ and $V$ with the corresponding graded subspaces of $\mathfrak{g}((t))$.

Let $C$ denote the category $D(G((t)))^{B^{-}[[t]]_{s_2}N((t))_{s_1},\psi}$. Then we have actions of $D(U)$ and $D(V)$ on $C$, and $\iota_{{s_1},{s_2},*}$ becomes the map

$$C^U\xrightarrow{\operatorname{Oblv}} C\xrightarrow{\operatorname{Av}_*} C^V.$$

As in Section 4, we consider the grading $\mathfrak{g}=\oplus\mathfrak{g}_i$ induced by $\operatorname{ad}{\check{\rho}}$. Then $\psi$ corresponds to an element of $\mathfrak{g}_1^*$. The Killing form $\langle -,-\rangle$ identifies $\mathfrak{g}_1^*$ with $\mathfrak{g}_{-1}$, so $\psi$ must come from an element $p_{-1}\in\mathfrak{g}_{-1}$. By the nondegeneracy of $\psi$, there is a (unique) choice of $p_1\in\mathfrak{g}_1$ so that $\{p_{-1},2\check{\rho},p_1\}$ is a $\mathfrak{sl}_2$ triple. We have a natural splitting 

$$\mathfrak{b}^-=\operatorname{ker}\operatorname{ad}{p_{-1}}\bigoplus[p_{1},\displaystyle\bigoplus_{i\leq 0}\mathfrak{g}_{i-1}].$$

Now define a subspace $U'\subseteq U$ by the pullback square

\begin{equation*}
\begin{tikzcd}
U' \arrow[r] \arrow[d] & U \arrow[d]\\
\lbrack p_{-1},\displaystyle\bigoplus_{i\leq 0}\mathfrak{g}_{i+1} \rbrack \arrow[r] & \mathfrak{b}^{-}.
\end{tikzcd}
\end{equation*}

\subsubsection{Proof schematic for fully faithfulness}

We claim that the map

$$C^{U'}\xrightarrow{\operatorname{Oblv}} C\xrightarrow{\operatorname{Av}_*} C^V$$

\noindent is an equivalence. In particular, our original map $C^U\rightarrow C^V$ is naturally equivalent to the map $C^U\rightarrow C^{U'}$, which is clearly fully faithful. To prove this, decompose $U'$ and $V$ into graded pieces $U'_i$ and $V_i$ as above. For $n$ a nonnegative integer, let $W_n$ be the vector space $U'_0\oplus\cdots\oplus U'_{1-n}\oplus V_{n+1}\oplus\cdots$. So for $n=0$ we have $W_0\cong V$, while for sufficiently large $n$ we have $W_n\cong U'$.

Let us construct an action of each $W_n$ on $C$. We already have actions of $U'$ and $V$, so we just need to show that the actions of $U'_0\oplus\cdots\oplus U'_{1-n}$ and $V_{n+1}\oplus\cdots$ commute. To prove this, let us work at the level of Lie algebras. We need to show that

$$[U'_0\oplus\cdots\oplus U'_{1-n},V_{n+1}\oplus\cdots]\subseteq\operatorname{ker}{\psi}\subseteq\mathfrak{b}^{-}[[t]]_{s_2}\mathfrak{n}((t))_{s_1}.$$

Inspecting the degrees, we see that $[U'_0\oplus\cdots\oplus U'_{1-n},V_{n+1}\oplus\cdots]$ must lie in degrees $2$ and above (with respect to $\operatorname{ad}\check{\rho}.$) In particular, if we can show that it is contained in $\mathfrak{n}((t))_{s_1}$, it will automatically be killed by $\psi$. Because $U'\subseteq\mathfrak{k}_{s_2\check{\rho},s_2-1}$ and $V\subseteq\mathfrak{p}_{s_2\check{\rho}}$, we have $[U',V]\subseteq\mathfrak{k}_{s_2\check{\rho},s_2-1}.$ Thus,

\begin{align*}
[U'_0\oplus\cdots\oplus U'_{1-n},V_{n+1}\oplus\cdots]&\subseteq\mathfrak{k}_{s_2\check{\rho},s_2-1}\cap\mathfrak{n}((t))\\
&\subseteq\mathfrak{k}_{s_2\check{\rho},0+}\cap\mathfrak{n}((t))\\
&\subseteq\mathfrak{p}_{s_1\check{\rho}}\cap\mathfrak{n}((t))\\
&=\mathfrak{n}((t))_{s_1}
\end{align*}

\noindent as desired.

Our strategy for relating $C^{U'}$ and $C^V$ will be to show that all the $C^{W_n}$ are isomorphic. As before, we have maps

$$C^{W_{n+1}}\xrightarrow{\operatorname{Oblv}} C^{W_{n+1}\cap W_n}\xrightarrow{\operatorname{Av}_*} C^{W_n}.$$

Let $D$ be the category $C^{W_{n+1}\cap W_n}$. Then we want to understand the map $D^{U'_{-n}}\rightarrow D^{V_{n+1}}.$ If we repeat the logic above, we see that the actions of $U'_{-n}$ and $V_{n+1}$ do not commute - instead, their commutation is controlled by the character $\psi$.

\subsubsection{Digression: the Heisenberg trick}

We will place our situation into a more general framework. 

\begin{proposition}\label{Heisenberg}
Let $X$ and $Y$ be vector spaces with a perfect pairing $X\times Y\rightarrow k$, and let $H$ be the corresponding Heisenberg group, with underlying variety isomorphic to $X\times Y\times\mathbb{G}_a.$ Take $C$ a category acted on by $H$. Assume that the center $Z\cong\mathbb{G}_a$ of $H$ acts on $C$ via a nontrivial character $\psi$. Then for any category $C$ acted on by $H$, the composition

\[C^X\rightarrow C\rightarrow C^Y\]

\noindent is an equivalence.
\end{proposition}

The proof will be delayed until the end of this section.

\subsubsection{The Heisenberg trick: application}

More precisely, let $H$ be the space $U'_{-n}\oplus V_{n+1}\oplus\mathbb{G}_a$. We endow $H$ with a (not necessarily commutative) group structure, which makes it a central extension of the commutative group $U'_{-n}\oplus V_{n+1}$. The only nontrivial commutators will come from the map $[-,-]_H: U'_{-n}\times V_{n+1}\mapsto\mathbb{G}_a$, which we define to equal $\psi([-,-]_{N((t))}).$ The actions of $U$ and $V$ extend to an action of $H$ on $C$, for which $C^{\mathbb{G}_a,\psi}\cong C.$

We claim that $[-,-]_H$ is a perfect pairing between $U'_{-n}$ and $V_{n+1}$. To see this, note that the maps from $U$ and $V$ to $\mathfrak{g}$ identify them with the subspaces $\displaystyle\bigoplus_{\substack{i\leq 0\\ (1-i)s_2\in\mathbb{Z}}}\mathfrak{g}_i$ and $\displaystyle\bigoplus_{\substack{i\geq 1\\ is_2\in\mathbb{Z}}}\mathfrak{g}_i$, respectively.

Therefore, if $(n+1)s_2\not\in\mathbb{Z}$, then both $U'_{-n}$ and $V_{n+1}$ are trivial and there is nothing to prove. On the other hand, if $(n+1)s_2\in\mathbb{Z},$ we can identify $U'_{-n}$ with $[p_{1},\mathfrak{g}_{-n-1}]$ and $V_{n+1}$ with $\mathfrak{g}_{n+1}$. Then we have
\begin{align*}
[[p_1,a],b]_H&=\langle p_{-1},[[p_1,a],b]\rangle\\
&=-\langle [p_{-1},b],[p_1,a]\rangle.
\end{align*}

As the Killing form is invariant under the $\mathfrak{sl}_2$ action, this is a perfect pairing, as desired. Thus, by Proposition \ref{Heisenberg}, the map $D^{U'_{-n}}\rightarrow D^{V_{n+1}}$ is an equivalence, as desired.

\subsubsection{Depth of the quotient}

It remains to show that the quotient of $C^{U}\rightarrow C^{U'}$ is purely of depth $s_2-1$. We will do so by relating said quotient to $D(G((t))/K_{s_2\check{\rho},(s_2-1)+})^{\circ},$ which we know to be purely of depth $s_2-1$ by Theorem \ref{generation}.

Note that $N((t))_{s_1}=P_{s_1\check{\rho}}\cap N((t))$ contains $K_{s_2\check{\rho},s_2-1}\cap N((t)).$ We denote by $E$ the category $D(G((t))/B^{-}[[t]]_{s_2})^{(K_{s_2\check{\rho},s_2-1}\cap N((t)),\psi)}$. Then $C$ can be expressed as a relative tensor product with $E$, so it suffices to show that the quotient of $E^{U}\rightarrow E^{U'}$ is purely of depth $s_2-1$.

For ease of notation, we will identify $K_{s_2\check{\rho},s_2-1}/K_{s_2\check{\rho},(s_2-1)+}$ with a subspace of $\mathfrak{g}$. Now we can rewrite all categories in sight in terms of $D(G((t))/K_{s_2\check{\rho},(s_2-1)+})$:

\begin{align*}
E&\cong D(G((t))/K_{s_2\check{\rho},(s_2-1)+})\otimes_{D((K_{s_2\check{\rho},s_2-1}/K_{s_2\check{\rho},(s_2-1)+})^*)}D(\psi+\mathfrak{n}^{\perp})\\
E^U&\cong D(G((t))/K_{s_2\check{\rho},(s_2-1)+})\otimes_{D((K_{s_2\check{\rho},s_2-1}/K_{s_2\check{\rho},(s_2-1)+})^*)}D(\psi)\\
E^{U'}&\cong D(G((t))/K_{s_2\check{\rho},(s_2-1)+})\otimes_{D((K_{s_2\check{\rho},s_2-1}/K_{s_2\check{\rho},(s_2-1)+})^*)}D(\psi+(\mathfrak{n}+U')^{\perp}).
\end{align*}

Comparing with Theorem \ref{generation}, we see that it suffices to show that

$$(\psi+(\mathfrak{n}+U')^{\perp})\backslash\psi\subseteq(K_{s_2\check{\rho},s_2-1}/K_{s_2\check{\rho},(s_2-1)+})^{*,\circ}.$$

We will show that the image of $(\psi+(\mathfrak{n}+U')^{\perp})\backslash\psi$ in $\mathfrak{g}^*$ is $G$-semistable, which will imply the desired statement. Let us dualize everything via the Killing form. Then $\psi$ becomes $p_{-1}$, and $(\mathfrak{n}+U')^{\perp}$ becomes $\operatorname{ker}\operatorname{ad}p_1\cap(K_{s_2\check{\rho},s_2-1}/K_{s_2\check{\rho},(s_2-1)+})^*\subseteq\operatorname{ker}\operatorname{ad}p_1.$ But the theory of the Kostant section tells us that the map $p_{-1}+\operatorname{ker}\operatorname{ad}p_1\rightarrow\mathfrak{g}//G$ is an isomorphism sending $p_{-1}$ to $0$, so
$$(p_{-1}+\operatorname{ker}\operatorname{ad}p_1\cap(K_{s_2\check{\rho},s_2-1}/K_{s_2\check{\rho},(s_2-1)+})^*)\backslash p_{-1}\subseteq(p_{-1}+\operatorname{ker}\operatorname{ad}p_1)\backslash p_{-1}$$

\noindent is $G$-semistable, as desired.

\subsubsection{The Heisenberg trick: proof}
\begin{proof}
Let us start with a reduction. We have:

\begin{align*}
C&\cong C^{Z,\psi}\\
&\cong(C\underset{D(H)}{\otimes}D(H))^{Z,\psi}\\
&\cong C\underset{D(H)}{\otimes}D(H)^{Z,\psi}.
\end{align*}

As the functor of taking invariants with respect to $U$ (or $V$) commutes with tensor products, this reduces us to the case of $C\cong D(H)^{Z,\psi}$. We need to show that the map

\[D(H/X)^{Z,\psi}\rightarrow D(H)^{Z,\psi}\rightarrow D(H/Y)^{Z,\psi}\]

\noindent is an equivalence. Label the maps of interest as in the following diagram:

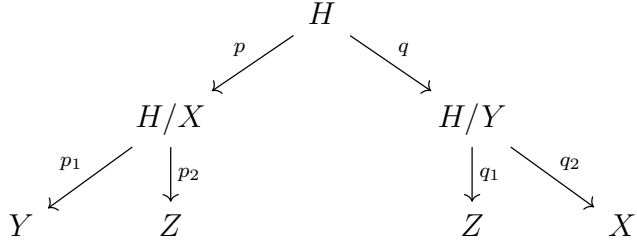
\begin{equation*}
\begin{tikzcd}
& & H \arrow{ld}[swap]{p} \arrow{rd}{q}& &\\
& H/X \arrow{ld}[swap]{p_1} \arrow{d}{p_2} & & H/Y \arrow{d}{q_1}\arrow{rd}{q_2} &\\
Y & Z & & Z & X\\
\end{tikzcd}
\end{equation*}

\noindent where we use the canonical isomorphisms $H/X\cong (Y\times Z\times X)/X\cong Y\times Z$ and $H/Y\cong Z\times X.$ These identifications also show that $D(H/X)^{Z,\psi}$ and $D(H/Y)^{Z,\psi}$ are canonically equivalent to $D(Y)$ and $D(X)$. Identifying $\psi$ with the corresponding character D-module on $Z$, we wish to show that the composite

\[D(Y)\xrightarrow{p_1^*(-)\otimes p_2^*\psi} D(H/X)\xrightarrow{p^*} D(H)\xrightarrow{q_*} D(H/Y)\xrightarrow{q_{2*}(-\otimes q_1^*(-\psi))} D(X)\]

\noindent is an equivalence. Applying the projection formula, we see that this can be rewritten as the compositon

\[D(Y)\xrightarrow{(p_1\circ p)^*}D(H)\xrightarrow{\otimes\mathcal{F}}D(H)\xrightarrow{(q_2\circ q)_*}D(X)\]

\noindent with $\mathcal{F}\cong((p_2\circ p)-(q_1\circ q))^*\psi.$ By inspection, we see that $\mathcal{F}$ is invariant under $Z$, and that the corresponding D-module on $H/Z\cong X\times Y$ is the pullback of $\psi$ along $[-,-]_H:X\times Y\rightarrow Z.$ But this is the kernel of a Fourier transform between $D(X)$ and $D(Y)$, which is an equivalence as desired.
\end{proof}
\end{proof}
 
 \subsection{Supplements}
 \subsubsection{Comparison of $\iota_{s_1,s_2,*}$ and $\iota_{s_1,s_2,!}$}
 In \cite{Whit}, they introduce, for $s_1$ and $s_2$ integral, both our functor $\iota_{s_1,s_2,*}$ and another functor $\iota_{s_1,s_2,!}$. (Note that our definition of $\iota_{s_1,s_2,*}$ makes sense even when $s_1$ and $s_2$ are not consecutive.) They then prove that these two functors coincide up to a shift. This result is crucial for their proof of the equivalence between the operations of Whittaker invariants and Whittaker coinvariants.
 
 Let us generalize this to our setting. Our functor $\iota_{s_1,s_2,!}$ will actually be a shift of the one considered in \cite{Whit}. Because of this shift, our $\iota_{s_1,s_2,*}$ and $\iota_{s_1,s_2,!}$ will coincide on the nose.
 
 Define $\iota_{s_1,s_2,!}$ as the (partially defined) composition
 
 \[D(G((t)))^{B^{-}[[t]]_{s_1}N((t))_{s_1},\psi}\xrightarrow{\operatorname{Oblv}}D(G((t)))^{B^{-}[[t]]_{s_2}N((t))_{s_1},\psi}\xrightarrow{\operatorname{Av}_!}D(G((t)))^{B^{-}[[t]]_{s_2}N((t))_{s_2},\psi}.\]
 
 A priori, this functor may only be partially defined, as the functor $\operatorname{Av}_!$ is only partially defined. However, we have:
 
 \begin{theorem}
 The functor $\iota_{s_1,s_2,!}$ is defined on all of $D(G((t)))^{B^{-}[[t]]_{s_1}N((t))_{s_1},\psi}$ and is equivalent to the functor $\iota_{s_1,s_2,*}$.
 \end{theorem}
 \begin{proof}
 We use the notation of the proof of Theorem \ref{consecutivejumploci}. The map $\iota_{s_1,s_2,!}$ is then the composition
 
 \[C^{U}\xrightarrow{\operatorname{Oblv}} C\xrightarrow{\operatorname{Av}_!} C^V.\]
 
 As for Theorem \ref{consecutivejumploci}, we will prove that the composition
 
 \[C^{U'}\xrightarrow{\operatorname{Oblv}} C\xrightarrow{\operatorname{Av}_!} C^V\]
 
 \noindent is an equivalence, and in fact is itself equivalent to the composition
 
 \[C^{U}\xrightarrow{\operatorname{Oblv}} C\xrightarrow{\operatorname{Av}_*} C^V.\]
 
 Tracing through the proof of Theorem \ref{consecutivejumploci}, we see that this boils down to showing the following extension of Proposition \ref{Heisenberg}.
 
 \begin{proposition}\label{Heisenberg2}
Let $X$ and $Y$ be vector spaces with a perfect pairing $X\times Y\rightarrow k$, and let $H$ be the corresponding Heisenberg group, with underlying variety isomorphic to $X\times Y\times\mathbb{G}_a.$ Take $C$ a category acted on by $H$. Assume that the center $Z\cong\mathbb{G}_a$ of $H$ acts on $C$ via a nontrivial character $\psi$. Then for any category $C$ acted on by $H$, the compositions

\[C^X\rightarrow C\xrightarrow{\operatorname{Av}_*} C^Y\]
\[C^X\rightarrow C\xrightarrow{\operatorname{Av}_!} C^Y\]

\noindent are equivalent.
\end{proposition}
\begin{proof}
As in the proof of Proposition \ref{Heisenberg}, we can reduce to the case of $C\cong D(H)^{Z,\psi}.$ Let $pr_1$ and $pr_2$ be the projections from $X\times Y$ to $X$ and $Y$, respectively. Recall that we identified the first composition (with $\operatorname{Av}_*$) with the Fourier transform

\[D(Y)\xrightarrow{pr_1^*}D(X\times Y)\xrightarrow{\otimes\mathcal{F}'}D(X\times Y)\xrightarrow{pr_{2,*}}D(X)\]

\noindent with $\mathcal{F}'$ the pullback of $\psi$ along $[-,-]_H:X\times Y\rightarrow Z.$ A similar argument identifies the second composition with

\[D(Y)\xrightarrow{pr_1^*}D(X\times Y)\xrightarrow{\otimes\mathcal{F}'}D(X\times Y)\xrightarrow{pr_{2,!}}D(X).\]

Taking orthogonal bases for $X$ and $Y$, we can reduce to the case where they are both $1$-dimensional. Compactify $Y$ to $\overline{Y}\cong\mathbb{P}^1,$ and let $j$ denote the open immersion $X\times Y\rightarrow X\times\overline{Y}.$ As $X\times\overline{Y}\rightarrow X$ is proper, it suffices to show that the compositions

\[D(Y)\xrightarrow{pr_1^*}D(X\times Y)\xrightarrow{\otimes\mathcal{F}'}D(X\times Y)\xrightarrow{j_*}D(X\times\overline{Y})\]

\noindent and

\[D(Y)\xrightarrow{pr_1^*}D(X\times Y)\xrightarrow{\otimes\mathcal{F}'}D(X\times Y)\xrightarrow{j_!}D(X\times\overline{Y})\]

\noindent coincide. As they are tautologically equivalent on $X\times Y$, it suffices to show that they coincide when restricted to $X\times(\overline{Y}\backslash 0)$.

Choose an identification of $Z$ with $\mathbb{G}_a$. Then the preimage under $[-,-]_H$ of $1\in Z$ defines an isomorphism between $X\backslash 0$ and $Y\backslash 0$, which extends to an isomorphism $X\cong\overline{Y}\backslash 0.$ We can define an action of $\mathbb{G}_a$ on $X\times(\overline{Y}\backslash 0)\cong X\times X$ by $a(x_1,x_2)=(x_1+ax_2,x_2).$

Let $\mathcal{G}$ be an object of $D(Y)$. Note that $pr_1^*\mathcal{G}\otimes\mathcal{F}'|_{X\times(Y/0)}$ is naturally $(\mathbb{G}_a,\psi)$-equivariant, and thus so is $j_*(pr_1^*\mathcal{G}\otimes\mathcal{F}')|_{X\times(\overline{Y}\backslash 0)}.$ We must show this object satisfies the defining property of $j_!$, i.e., that there are no nontrivial morphisms from it to any objects supported on $X\times\infty\subset X\times(\overline{Y}\backslash 0)$.

Let $i$ be the inclusion $X\times\infty\rightarrow X\times(\overline{Y}\backslash 0),$ and take $\mathcal{G}'$ some D-module on $X\times\infty.$ Any morphism $j_*(pr_1^*\mathcal{G}\otimes\mathcal{F}')|_{X\times(\overline{Y}\backslash 0)}\rightarrow i_*\mathcal{G}'$ must factor through $\operatorname{Av}_*^{\mathbb{G}_a,\psi}i_*\mathcal{G}\cong i_*\operatorname{Av}_*^{\mathbb{G}_a,\psi}\mathcal{G}.$ But the action of $\mathbb{G}_a$ on $X\times\infty$ is trivial, so $\operatorname{Av}_*^{\mathbb{G}_a,\psi}\mathcal{G}$ is the zero object. Thus all such morphisms must be trivial, as desired.
\end{proof}
 \end{proof}
 
 \subsubsection{Generalized vacuum modules}
 
 After introducing the adolescent Whittaker construction, \cite{Whit} uses it to study a sequence of modules $\mathcal{W}^n_{\kappa}$ over the $\mathcal{W}$-algebra at level $\kappa$. Let $\Psi$ be the functor of Drinfeld-Sokolov reduction. Their Theorem 4.5.1 shows that the complex
 
 \[\Psi(\operatorname{ind}_{\operatorname{Lie}\overset{\circ}{I}_n}^{\hat{\mathfrak{g}}_{\kappa}}\psi_{\overset{\circ}{I}_n})\]
 
\noindent is supported in exactly one cohomological degree, and they define $\mathcal{W}^n_{\kappa}$ to be the module in that degree.
 
 It thus seems likely that, for $r$ a nonnegative rational number, the complex
 
\[\Psi(\operatorname{ind}_{\operatorname{Lie}\overset{\circ}{I}_{r+1}}^{\hat{\mathfrak{g}}_{\kappa}}\psi_{\overset{\circ}{I}_{r+1}})\]

\noindent is also supported in one cohomological degree. Furthermore, the module $\mathcal{W}^{r+1}_{\kappa}$ in that degree should be equivalent to the one constructed in Remark 4.4.3 of \cite{gurbirsam}.

It seems plausible that this can be proven with the methods of \cite{Whit}. However, our methods are less explicit and do not seem applicable to this problem, and so we do not pursue this direction further.
 
 \appendix
 
 \section{Appendix}
 
 Here we prove a few simple facts about D-modules for which we could not find references in the literature.
 
 \begin{lemma}\label{appendixgen}
 Let $X$ and $Y$ be varieties over $k$, and let $f:X\rightarrow Y$ be a morphism surjective on $k$-points. Then the image of $f_*$ generates $D(Y)$.
 \end{lemma}
 
 \begin{proof}
By Noetherian induction, we can assume that for every proper closed subvariety $Z$ of $Y$, the image of pushforward along $f_Z: X\times_Y Z\rightarrow Z$ generates $D(Z)$. It then suffices to prove the Lemma after base changing along a dense open subvariety $U$ of $Y$. Indeed, letting $Z$ be the complement of $U$ inside $Y$, we see that the image of $f_*$ contains both $D(U)$ and $D(Z)$, which together generate $D(X)$. 

In particular, we can assume that $Y$ is irreducible. Let $\eta_Y$ denote its generic point. Choose any closed point of $X\times_Y \eta_Y$. Its closure in $X$ will be a subvariety $S$ of $X$ so that $g:S\rightarrow Y$ is dominant and generically finite. Passing to an open subset of $Y$, we can assume that $S$ is finite over $Y$.

It now suffices to show that $g_*$ has generating image, or equivalently, that $g^!$ is conservative. But this follows from proper descent for D-modules.
 \end{proof}
 
 \begin{lemma}\label{appendixtensor}
 Let $X$ be a smooth variety over $k$ and $U\subseteq X$ an open subvariety. Consider $D(X)$ as a symmetric monoidal category via the $!$-tensor product and give $\QCoh(X)$ the module structure corresponding to the realization of $D(X)$ via left D-modules. Then
 
 \[\QCoh(X)\otimes_{D(X)}D(U)\cong\QCoh(U). \]
 \end{lemma}
 
 \begin{proof}
 Let $A\in D(X)$ be the object $j_*\omega_U$, where $j$ is the inclusion $U\rightarrow X$ and $\omega_U$ is the dualizing D-module on $U$. Proceeding as in the proof of Lemma I.3.3.2.4. of \cite{GaRo}, we can show that $D(U)\cong A-\operatorname{mod}(D(X))$ and $\QCoh(U)\cong A-\operatorname{mod}(\QCoh(U))$. Corollary I.1.8.5.7. of \cite{GaRo} now gives the desired equivalence.
 \end{proof}
 
\printbibliography

\end{document}